\documentclass[11pt]{article}
\usepackage[T1]{fontenc}
\usepackage[latin1]{inputenc}
\usepackage{amsmath,amssymb,euscript}
\usepackage{bbm}
\usepackage{graphicx}
\usepackage{epic}
\usepackage{epsfig}
\usepackage{pstricks}
\usepackage{psfrag}
 \usepackage{curves}
\usepackage{mathrsfs}

\newtheorem{thm}{Theorem}[section]

\newtheorem{cor}[thm]{Corollary}

\newtheorem{defi}[thm]{Definition}
\newtheorem{prop}[thm]{Proposition}
\newtheorem{rem}[thm]{Remark}
\newtheorem{hyp}[thm]{Assumption}

\newenvironment{proof}{\noindent {\bf Proof \phantom{9}}}
{\hfill $\square$ \vspace{0.25cm}}

\def\be{\begin{eqnarray}}
\def\ee{\end{eqnarray}}
\def\ben{\begin{eqnarray*}}
\def\een{\end{eqnarray*}}

\numberwithin{equation}{section}
\numberwithin{figure}{section}

\def\be{\begin{eqnarray}}
\def\ee{\end{eqnarray}}

\def\me{\medskip\noindent}
\def\bi{\bigskip\noindent}

\newcommand{\Co}{\mathcal{C}}

\newcommand{\card}{\mbox{Card}}

\newcommand{\supp}{\mbox{supp}}
\def\D{\mathbb{D}}

\def\N{\mathbb{N}}
\def\P{\mathbb{P}}
\def\R{\mathbb{R}}
\def\E{\mathbb{E}}
\def\X{\mathcal{X}}
\def\U{\mathcal{U}}

\def\ind{{\mathchoice {\rm 1\mskip-4mu l} {\rm 1\mskip-4mu l}
{\rm 1\mskip-4.5mu l} {\rm 1\mskip-5mu l}}}

\title{\bf Stochastic dynamics of adaptive trait and neutral marker driven by eco-evolutionary feedbacks}

\author{Sylvain Billiard\thanks{Laboratoire GEPV, UFR de Biologie, Université des Sciences et Technologies Lille 1, Cité scientifique, 59655 Villeneuve d'Ascq Cedex-France; E-mail: \texttt{sylvain.billiard@univ-lille1.fr}}, \qquad Régis Ferrière\thanks{Eco-Evolution Mathématique, CNRS UMR 7625, Ecole Normale Supérieure, 46 rue d'Ulm, 75230 Paris, France; E-mail: \texttt{ferriere@biologie.ens.fr}},\qquad \qquad  Sylvie M\'el\'eard\thanks{CMAP, Ecole Polytechnique, CNRS, route de
    Saclay, 91128 Palaiseau Cedex-France; E-mail: \texttt{sylvie.meleard@polytechnique.edu}}, \qquad Viet Chi Tran\thanks{Laboratoire P. Painlevé, UFR de Mathématiques, Université des Sciences et Technologies Lille 1, Cité scientifique, 59655 Villeneuve d'Ascq Cedex-France; E-mail: \texttt{chi.tran@math.univ-lille1.fr}}}

\date{\today}

\begin{document}

\maketitle

\begin{abstract}How the neutral diversity is affected by selection and adaptation is investigated in an eco-evolutionary framework. In our model, we study a finite population in continuous time, where each individual is characterized by a trait under selection and a completely linked neutral marker. Population dynamics are driven by births and deaths, mutations at birth, and competition between individuals. Trait values influence ecological processes (demographic events, competition), and competition generates selection on trait variation, thus closing the eco-evolutionary feedback loop. The demographic effects of the trait are also expected to influence the generation and maintenance of neutral variation. We consider a large population limit with rare mutation, under the assumption that the neutral marker mutates faster than the trait under selection. We prove the convergence of the stochastic individual-based process to a new measure-valued diffusive process with jumps that we call Substitution Fleming-Viot Process (SFVP). When restricted to the trait space this process is the Trait Substitution Sequence first introduced by Metz et al. (1996). During the invasion of a favorable mutation, a genetical bottleneck occurs and the marker associated with this favorable mutant is hitchhiked. By rigorously analysing the hitchhiking effect and how the neutral diversity is restored afterwards, we obtain the condition for a time-scale separation; under this condition, we show that the marker distribution is approximated by a Fleming-Viot distribution between two trait substitutions. We discuss the implications of the SFVP for our understanding of the dynamics of neutral variation under eco-evolutionary feedbacks and illustrate the main phenomena with simulations. Our results highlight the joint importance of mutations, ecological parameters, and trait values in the restoration of neutral diversity after a selective sweep. \end{abstract}

\bigskip


\emph{Acknowledgements:} S.B., S.M. and V.C.T. have been supported by the ANR MANEGE (ANR-09-BLAN-0215), the Chair ``Modélisation Mathématique et Biodiversité" of Veolia Environnement-Ecole Polytechnique-Museum National d'Histoire Naturelle-Fondation X. V.C.T. also acknowledges support from Labex CEMPI (ANR-11-LABX-0007-01). R.F. acknowledges support from National Science Foundation Award EF-0623632, the Institut Universitaire de France, and the Agence Nationale de la Recherche ("EVORANGE" grant).

\bigskip

\section{Introduction}
\label{sec:intro}

	The science of biodiversity currently faces the challenge of understanding how ecological processes shape evolutionary change, and reciprocally how evolution affects the structure and function of ecological systems (Schoener \cite{schoener}). Such eco-evolutionary feedbacks determine the dynamics of so-called adaptive traits -quantitative characters that are heritable yet mutable from parent to offspring (Dieckmann Law \cite{dieckmannlaw}, Metz et al. \cite{metzgeritzmeszenajacobsheerwaarden}). Under the combined assumptions of large population and rare mutation scalings, the time evolution of an adaptive trait can be described as a sequence of mutant invasions, each being driven by positive selection in the ecological context set by the `resident' value of the adaptive trait (Metz et al. \cite{metznisbetgeritz}). The resulting evolutionary model is a jump process called the Trait Substitution Sequence (TSS): every new trait mutant either goes extinct, or replaces the resident, causing the TSS to jumps from the former resident population equilibrium to a new equilibrium (Metz et al. \cite{metzgeritzmeszenajacobsheerwaarden}, Champagnat \cite{champagnat06} and Champagnat et al. \cite{champagnatferrieremeleard2}). In population genetics, these jumps are known as selective sweeps (Barton \cite{barton98}, Stephan et al. \cite{stephanwiehelenz}). Previous works by \cite{lloyd,christiansenloeschcke,hammerstein96,weissing96,matessischneider,eshelfeldmann} support the view that the TSS as a model of long-term phenotypic evolution is relatively insensitive to the details of the genetic determination of the trait.

\me
	Whereas eco-evolutionary feedbacks can result in variation of adaptive traits among populations (and even within populations when evolutionary branching occurs, Geritz et al. \cite{geritzkisdimeszenametz}), much of the molecular diversity measured by population geneticists involve DNA sequences of no known adaptive value, i.e. selectively neutral. A neutral sequence that is physically linked in the genome to the sequence that codes for the adaptive trait is called a marker of that trait. A longstanding question in evolutionary theory is understanding how variation in such molecular markers evolves, and how patterns of neutral molecular evolution can be used to infer the history of trait mutation that have driven past adaptation.
	
	When adaptive mutations are rare, adaptation proceeds as a series of selective sweeps: a trait mutation occurs while the population is monomorphic for the trait, and increases rapidly in frequency toward fixation. Following on from Kojima and Schaffer \cite{kojimaschaffer}, Maynard Smith and Haigh \cite{maynardsmithhaigh} pointed out that selective sweeps purge genetic variation at linked sites: a particular marker allele goes to fixation as a consequence of linkage with the selected allele, a phenomenon they dubbed the `hitchhiking effect'. Maynard Smith and Haigh's deterministic model was revisited in a stochastic approach by Ohta and Kimura \cite{ohtakimura}. These seminal studies of hitchhiking focused on the short-term dynamics of an interaction between two alleles at the locus under selection and two alleles at the neutral locus. Long-term dynamics were considered first by Kaplan et al. \cite{kaplan} who developed a stochastic model for finite populations to describe the effect of recurrent hitchhiking. In order to describe stationary levels of nucleotide diversity at the marker locus, they used the infinite site model and a coalescent approach under the assumption of constant population size and constant selection coefficients. This has generated an abundant theoretical literature on modeling the impact of selection on neutral polymorphism (Barton \cite{barton}, Etheridge et al. \cite{etheridgepfaffelhuberwakolbinger}, Durrett-Schweinsberg \cite{durrettschweinsberg} and references therein). Recent deterministic models have relaxed the assumption of constant selection either because of the presence of genetic backgrounds (e.g. assuming a quantitative trait \cite{chevinhospital}) or in the case of a parasite, because of the complexity of the demographic events involved in the life cycle \cite{schneiderkim}. All previous models assume constant population size and constant selection, or that the population size is independent of the selective value of the individuals. \\

	In this article, our goal is to relax these key assumptions. Under general ecological scenarios, eco-evolutionary feedbacks operate: as the adaptive trait evolves, population size and selection co-vary. The eco-evolutionary process of adaptive trait and neutral marker dynamics requires a rigorous mathematical framework, the foundation of which we establish here.
We start with a `microscopic', individual-based model where individuals have two heritable characteristics: (i) an adaptive trait that influences their intrinsic demographic rates and ecological interactions, and (ii) a genetic marker that has no demographic or ecological effects, hence, is selectively neutral. This work focuses entirely on asexual populations and short genomic regions that remain perfectly linked to the loci under selection, neglecting recombination. 
The population is described by a measure according to which each individual is represented by a Dirac mass that weights its characters. This leads to study the population eco-evolutionary dynamics as a measure-valued stochastic process.

The dynamics are driven by competition between individuals, asexual reproduction without or with mutation, and death. Variation in population size and selection as the trait evolves are mediated by the demographic effects of change in the trait. These effects are expected to influence the generation and maintenance of neutral variation.

	The effect of mutation on the marker can be continuous or discrete. Our framework thus encompasses a variety of conventional mutation models such as the two-alleles model, the stepwise mutation model, and the continuous state mutation model. Our distinctive assumption here is that the marker mutation process is much faster than the trait mutation process but much slower than the ecological time-scale of birth and death events. This is supported by the fact that most mutations are neutral or nearly neutral (such as mutations involved in microsatellite variation). Therefore, there are three time scales in the model: the fast ecological time scale of birth and death events, the slow time scale of trait mutation, and an intermediate time scale of marker mutation. We study the joint process of trait and marker dynamics on the trait mutation time scale.
	
	We are interested in limit theorems when the population carrying capacity goes to infinity. Then, the population size stabilizes in a neighborhood of the ecological equilibrium and jumps to another equilibrium when a successful trait mutant goes to fixation in the population. This is the TSS dynamics of the adaptive trait. It does not depend on the marker and has been mathematically proved by Champagnat \cite{champagnat06}. The novelty in the model and in the proofs come from the time-scales difference for the marker and trait mutations.
The study of the marker distribution during the invasion period requires careful consideration of the individual process and of the different scales involved. In a first period, starting with the single invading mutant, we prove that the marker distribution remains close to a Dirac mass at the value of the initial mutant. Until the next jump of the TSS, the marker evolves as a stochastic distribution-valued process. In the case where the marker mutation effects are continuous and small, this is a Fleming-Viot process whose drift and covariance depend on the resident adaptive trait. In every cases, for any marker mutation model, the collated dynamics define a measure-valued diffusive process with jumps that we call Substitution Fleming-Viot Process (SFVP). The convergence of the microscopic process to the SFVP is shown both in the sense of finite dimensional distributions and in the sense of occupation measure, thus improving previous results of Champagnat \cite{champagnat06}.
	
	From a biological standpoint, we recover the conventional hitchhiking phenomenon: when a new mutant trait appears and sweeps through the population to fixation, the marker carried by the mutant individual is hitchhiked, and the marker distribution undergoes a genetical bottleneck. The mathematical construction of the SFVP process has new implications of biological relevance. Neutral diversity is restored after each adaptive jump, but as the adaptive trait evolves, population size, the mutation rate, genetic drift and demographic fluctuations change, which causes the rate of neutral polymorphism build-up and the moments of the marker distribution to change too. This suggests that the nature and structure of the whole eco-evolutionary feedback loop (i.e. how adaptive traits influence demographic rates and ecological interactions, and how ecological processes shape selection pressures on adaptive traits) may be important to explain the extreme disparities in genetic neutral diversity observed among species, even closely related ones and in the absence of differences in recombination profiles (Cutter and Payseur \cite{cutterpayseur}). In fact, it is well-known that demographic differences due to external causes (demographic bottleneck or population expansion due to environmental changes) can affect neutral diversity of a population and that closely related species can show very different neutral diversity patterns. Here, we show that internal causes of demographic variation involved in adaptation can also affect species differently.
	
	The article is organized as follows. In Section \ref{sec:IBM}, we start with the model description. The  stochastic individual-based process and its key assumptions are carefully described and examples are provided. A key parameter is $K$, an integer that gives the order of the population size and is used to rescale the mutation rates and kernels. By letting $K$ go to infinity we study the large population limit of the stochastic process. The main theorem is enounced and discussed in Section \ref{section:theorem}, where biological implications are also highlighted.
Time scale separations implied by the dependence in $K$ of the trait and marker mutations lead to homogenization phenomena and then to the SFVP. Our mathematical analysis provides a precise description of the genetical bottleneck that occurs at each trait substitution. We show that the marker of the initial mutant individual dominates in the marker distribution of the mutant population until this population reaches a neighborhood of the new ecological equilibrium.
Then, we present two numerical examples based on an ecological model adapted from Dieckmann and Doebeli \cite{dieckmanndoebeli}. In the first example, marker mutation is described by a continuous state model that leads to a piecewise Fleming-Viot process (section \ref{section:DD}) for the marker. In the second example, marker mutation follows a discrete two-allele model; then the classical Wright-Fisher diffusion \eqref{section:coexistence} is recovered. Further generalizations are discussed.
The proof of the main theorem in the adaptive dynamics scaling is in Section \ref{section:proof}. After having introduced a semi-martingale decomposition of our stochastic measure-valued process, we start with recalling and refining the result of Champagnat \cite{champagnat06} for the convergence of trait-marginals. For this purpose, we introduce the M1-topology on the Skorokhod space where the TSS lives, using some ideas of Collet et al. \cite{CMM}. This allows us to obtain the convergence to the TSS for the topology of occupation measure, hence providing additional pathwise information that complement the results of \cite{champagnat06}. The second part of the proof focuses on the marker distribution in an invading mutant population. This gives the result on the genetical bottleneck. Then, between two trait substitutions, the dynamics of the marker converges to a diffusive measure-valued process. As a conclusion of the proof, we show the convergence to the SFVP for the topology of occupation measures.

\section{The stochastic model}
\label{sec:IBM}

We consider an asexual population driven by births and deaths where each individual is characterized by hereditary types: a phenotypic trait under selection and a neutral marker. The trait and marker spaces ${\cal X}$ and ${\cal U}$ are assumed to be
compact subsets of $\mathbb{R}$.  The type of individual $i$ is thus a pair $(x_i,u_i)$, $x_{i}\in {\cal X}$ being the trait
value and $u_i\in {\cal U}$ its neutral marker.  The individual-based microscopic model from which we start is a stochastic birth and death process with density-dependence whose demographic parameters are functions of the trait under selection and are independent of the marker. We assume that the population size scales with an integer parameter $K$ tending to infinity while individuals are weighted with $\frac{1}{K}$.
At any time $t\geq 0$, we have a finite number $N^K_t$ of individuals, each of them holding trait and marker values in ${\cal X}\times {\cal U}$. Let us denote by $((x_{1}, u_{1}), \ldots, (x_{N^K_t},u_{N^K_t}))$ the trait and marker values of these individuals. The state of the population at time
$t\geq 0$, rescaled by $K$, is described by the point measure
\begin{equation}
  \label{eq:nu_t}
  \nu^{K}_t={1\over K}\sum_{{i=1}}^{N^K_{t}} \delta_{(x_{i},u_{i})},
\end{equation}
where $\delta_{(x,u)}$ is the Dirac measure at $(x,u)$. This measure belongs to the set of finite point measures on $\mathcal{X}\times {\cal U}$ with mass
$1/K$. This set is a subset of the set $\mathcal{M}_{F}({\cal X}\times {\cal U})$  of finite measures on ${\cal X}\times {\cal U}$, which is embedded by the weak convergence topology.
We denote by $\: \langle\nu,f\rangle$  the integral of the measurable function $f$ with
respect to the measure $\nu$ and  by $\mathrm{Supp}(\nu)$  the support of $\nu$.  Then $\: \langle\nu^{K}_t,{\bf 1}\rangle=\frac{N^K_t}{K}$.

\me For any $t\geq 0$, we also introduce  the trait marginal of the measure $\nu^K_t$  on $\mathcal{X}$, denoted by $ X^{K}_t$ and defined by
$$X^{K}_t = {1\over K}\sum_{{i=1}}^{N^K_{t}} \delta_{x_{i}}.$$ Therefore, the population measure $\nu^K_{t}$ writes
\be
\label{dec-population} \nu_t^{K}(dx,du)= X^K_{t}(dx)\,\pi^K_{t}(x,du)\ee
where $ \pi^K_{t}(x,du)$ is the marker distribution for a given trait value $x$
defined  by
\be
\label{def-pi}
\pi^K_{t}(x,du) = \frac{\sum_{i=1}^{N^K_t} \ind_{x_{i}=x}\delta_{u_i} }{\sum_{i=1}^{N^K_t} \ind_{x_{i}=x}}.
\ee

\bi
Our purpose is to study the asymptotic behavior  of the measure-valued process $\nu^K$ at large times, when the trait and marker are inherited but mutations occur. The main interest of our model is that these mutations happen at different time scales for trait and marker, but both longer than the individuals lifetime scale. The trait mutates much slower than the marker and drives the evolution time scale. Thus, the limiting behavior results from the interplay of three time scales: births and deaths, trait mutations and marker mutations.

\bi
We describe the individuals' life history. The trait has an influence on the ability of individuals to survive (including competition with other ones) and to reproduce but the marker is neutral. The demographic parameters are thus functions of the trait only and are defined on ${\cal X}$.

\begin{hyp}
  \label{hypo}
 \medskip
\begin{itemize}
\item An individual with trait $x$ and marker $u$ reproduces with birth rate given by
 $
   0\leq  b(x)\leq \bar b$, the function $ b$ being continuous.

\item Reproduction produces a single offspring which usually inherits the trait and marker of its ancestor except when a mutation occurs. Mutations on trait and marker occur independently with probabilities $p_{K}$ and $q_{K}$ respectively. Mutations are rare and  the marker mutates much more often than the trait. We assume that
\be
\label{echelle} q_{K} = p_{K} \, r_K\ , \quad \hbox{ with }  p_{K}=\frac{1}{K^2}, \quad q_{K} \to_{K\to \infty} 0\ , \quad r_K \to_{K\to \infty} + \infty.
\ee

  \item When a trait mutation occurs, the new trait of the descendant  is $x+k\in {\cal X}$ with $k$ chosen according to the probability measure $m(x,k)dk$.

\item  When a marker mutation occurs, the new marker of the descendant  is $u+h\in {\cal U}$ with $h$ chosen according to  the probability measure $G_{K}(u,dh)$. 

  \noindent For any $u\in \U$, $G_{K}(u,.)$ 
  is approximated as follows when $K$ tends to infinity:
    \be
   \label{hyp-D}
\lim_{K\rightarrow +\infty}\sup_{u\in \U} \bigg | {r_K\over K} \int_{ \U} (\phi(u+h)-\phi(u)) G_{K}(u,dh) - A\phi\bigg | =0,
  \ee
where  $(A, {\cal D}(A))$ is  the generator of a Feller semigroup and $\phi\in {\cal D}(A)\subseteq \mathcal{C}_b(\U,\R)$, the set of continuous bounded real functions on $\U$.


\item An individual with trait $x$ and marker $u$ dies with intrinsic death rate $0\leq  d(x)\leq \bar d$, the function $ d$ being continuous. Moreover the individual experiences competition the effect of which is an additional death rate $\,\eta(x)\ C*\nu^K_{t}(x)= \frac{\eta(x)}{K}\sum_{i=1}^{N^K_t}C(x-x_{i})$. The quantity $C(x-x_{i})$ describes the competition pressure exerted by an individual with trait $\,x_{i}$ on an individual with trait $x$.
 We assume that the functions $C$ and $\eta$ are continuous and that there exists  $\underline{\eta}>0$ such that
\be
\label{hyp-competition} \forall x, y \in {\cal X},\quad  \eta(x)\ C(x-y)\geq  \underline{\eta}>0.\ee

\end{itemize}
  \end{hyp}

\bigskip

\me
\noindent A classical choice of competition function $C$ is $C\equiv 1$ which is called ``mean field case" or ``logistic case". In that case the competition death rate is $\eta(x) N^K_t/K$.


\medskip
\begin{rem} Let us insist on the generality of Assumption \eqref{hyp-D} which allows a larger set of possible dynamics.

\begin{itemize}
\item Equation \eqref{hyp-D} is for example true for ${\cal U}=[u_{1},u_{2} ]$, $G_{K}$  a  centered Gaussian law  (conditioned to ${\cal U}$) with variance $\sigma_{K}\to 0$ such that $\lim_{K}  \sigma_{K}^2  {r_K\over K} = \sigma^2$ and $A\phi={\sigma^2\over 2}\phi''$ for $\phi\in \Co^2$ with $\phi'(u_{1})=\phi'(u_{2}) =0$.

\noindent Choosing for example $r_K=K^{3/2}$, $q_K=1/\sqrt{K}$ and $\sigma^2_K=1/\sqrt{K}$ works. This choice can be seen as a continuous state generalization of the stepwise mutation model \cite{ohtakimura73}.

\item If in addition the distribution $G_{K}$ has a non zero mean $\, \mu_{K}$ such that $\,{r_K\mu_{K}\over K} \to \mu>0$ corresponding to a mutational directional drift, then
the operator $A$ will be defined by $A\phi = {\sigma^2\over 2}\phi'' + \mu \phi'$.



\item If we relax the compactness of  ${\cal U}$ and assume that   ${\cal U} = \mathbb{R}$, a third choice consists in taking for $G_{K}$ the law of a Pareto variable with index $\alpha\in (1,2)$ divided by  $K^{\eta/\alpha}$, for $\eta\in (0,1]$. Then it has been proved in Jourdain et al. \cite{JMW} that
$$\lim_{K} \sup_{u} \left|K^\eta \int_{  \mathbb{R}} (\phi(u+h)-\phi(u)) G_{K}(u,dh) - {\alpha\over 2} D^\alpha \phi(u)\right|=0,$$
where
$$ D^\alpha \phi(u) = \int_{  \mathbb{R}}(\phi(u+h)-\phi(u)- h\phi'(u)\ind_{|h|\leq 1}) {dh\over |h|^{1+\alpha}}$$ is the fractional Laplacian with index $\alpha$. Thus if we take $r_K$ such that ${r_K\over K^{1+\eta}}$ converges as $K$ tends to infinity, and choose $A=D^\alpha$ in \eqref{hyp-D}, Assumptions \eqref{echelle}-\eqref{hyp-D} will be satisfied as soon as $\eta<1$.

\item
Another very interesting case is the discrete case when $\U=\{a,A\}$ is a set of two alleles. The mutation kernel is given by
\begin{equation}
G_K(u,dv)=\ind_{u=a}\,q_a\, \delta_A(dv)+\ind_{u=A}\,q_A \,\delta_a(dv).\label{DK:bernoulli}
\end{equation}In this case, \eqref{hyp-D} implies that $r_K/K$ has a limit when $K\rightarrow +\infty$. Let $\bar{r}$ be this limit, then
\be
A\phi(u)=\bar{r}\Big(\ind_{u=a}\, q_a \big(\phi(A)-\phi(a)\big)+\ind_{u=A}\, q_A \big(\phi(a)-\phi(A)\big)\Big).\label{A:bernoulli}
\ee
\end{itemize}
\end{rem}
\medskip

\noindent We see that the ratio between the two mutation probabilities $r_K=q_K/p_K$ that allows convergence is highly dependent on the mutation distribution.\\

\bigskip
\noindent Note that since the demographic rates do not depend on the marker,
  the dynamics of the population distribution of the trait is independent of the marker distribution. But the dynamics of the marker distribution cannot be separated from the trait distribution as we shall see.


\bigskip
\noindent The process $(\nu^{K}_t,t\geq 0)$ is a càdlàg ${\cal M}_F(\X\times \U)$-valued Markov process.
Existence and uniqueness in law of the process can be adapted from \cite{fourniermeleard,champagnatferrieremeleard2} under the assumption that $\mathbb{E}(\langle \nu^{K}_0,\mathbf{1}\rangle)< + \infty$.\\
Moreover, Assumption \eqref{hyp-competition} allows to prove as  in Champagnat \cite[Lemma 1]{champagnat06} that if
 for $p\geq 1$, $\sup_{K\in \N^*}\E(\langle \nu^K_0,\mathbf{1}\rangle^p)<+\infty$, then
\begin{equation}
\sup_{t\in \R_+, K\in \N^*} \E\big((\langle \nu^K_t,\mathbf{1}\rangle)^p\big)<+\infty\label{estimeemoment}
\end{equation}which will be useful to study the tightness and convergence of the sequence.

\section{Convergence to the Substitution Fleming-Viot Process}\label{section:theorem}

\medskip
\noindent The adaptive trait mutation time scale is the slowest, equal to ${1\over K p_{K}}=K$ by Assumptions \eqref{echelle}. It scales the evolutionary time. So we shall consider the limiting behavior of $(\nu^K_{Kt}, t\geq 0)$. We will see in section \ref{sec:MarkDistrib} that $p_K$ of order $1/K^2$ is the only choice which leads to a non-trivial or non-degenerate marker dynamics.

\me
Before stating our main result, we introduce several important ingredients which are used to describe the limit of $(\nu^K_{Kt}, t\geq 0)$ when $K\rightarrow +\infty$. We conclude the section with extensions and simulations.

\subsection{Invasion fitness function}

The large population behavior of the process $(\nu^K_t, t\geq 0)$ as $K$ tends to infinity, can be studied by classical arguments and is given in the appendix. At the ecological time scale (of order $1$), no mutation occurs in the asymptotic $K\to +\infty$. If the initial population has a single adaptive  trait $x$, then, in the limit $K\rightarrow +\infty$, the trait distribution remains $\delta_x$ since $p_K$ and $q_K$ vanish in the limit. The rescaled population size process $( N^K_{t}/K, t\geq 0)$ converges to the solution $(n_{t}, t\geq 0)$ of the ordinary differential equation
\begin{equation}
\frac{dn_t}{dt} =\big(b(x)-d(x)-\eta(x) C(0)n_t\big) n_t \label{eq1:logistique}
\end{equation}
which converges when $t$ tends to infinity to the equilibrium
\begin{equation}
\widehat{n}_x =\frac{
b(x)-d(x)}{\eta(x)C(0)}.\label{nchap}
\end{equation}

\me
Conversely, at the adaptive trait-mutation time scale $Kt$, new mutant traits can invade. If they replace the previous traits, then the corresponding event is called ``fixation''.\\
The probability of fixation of a mutant trait $y$ in a trait resident population $x$ at equilibrium depends on the invasion fitness function $f(y;x)$:
\begin{equation}
f(y;x)= b(y)-d(y)-\eta(y)\,C(y-x)\,\widehat{n}_x.\label{def:fitness}
\end{equation}
This fitness function describes the initial growth of the mutant population. It does not depend on the neutral marker.\\

\noindent By simplicity we work under the assumption of `invasion implies fixation', but this assumption will be relaxed in Section \ref{section:coexistence}. When a mutant trait appears, either its line of descent replaces the resident population or it disappears. As a consequence, two traits can not coexist in the long term.
\begin{hyp}[``Invasion implies fixation"]
\label{IIF}For all $x\in \X$ and for almost every $y\in \X$,
\begin{eqnarray*}
&\mbox{either} &\frac{b(y) - d(y)}{\eta(y) C(y-x)} <\frac{b(x) - d(x)}{\eta(x)C(0)},\\
&\mbox{or}&
\frac{b(y) - d(y)}{\eta(y)C(y-x)}> \frac{b(x) - d(x)}{\eta(x)C(0)}\ \mbox{ and }\
\frac{b(x) - d(x)}{\eta(x)C(x-y)}<\frac{b(y) - d(y)}{\eta(y)C(0)}.
\end{eqnarray*}
\end{hyp}

\begin{rem}In the case of logistic populations with $C\equiv 1$, this assumption is satisfied as soon as $x \mapsto \widehat{n}_x$ is strictly monotonous.
\end{rem}




\subsection{Main theorem}

\bigskip \noindent
Let us first give  the definition of the Fleming-Viot process which will appear in our setting (see e.g. Dawson and Hochberg \cite{dawsonhochberg}, Dawson \cite{dawson}, Donnelly and Kurz \cite{donnellykurtz}, Etheridge \cite{etheridgebook}). 
We recall that the operator $A$ has been introduced in \eqref{hyp-D}.

\me In the sequel,  we denote by $\mathcal{P}(\U)$ and $\mathcal{P}(\X \times \U)$ the probability measure spaces respectively on $\U$ and on $\X \times \U$.

\begin{defi}\label{def:FV} Let us fix $x\in {\cal X}$ and $u \in {\cal U}$.
The Fleming-Viot process $(F_{t}^{u}(x,.), t\geq 0)$ indexed by $x$, started at time $0$ with initial condition $\delta_{u}$ and associated with the mutation operator $A$ is the  ${\cal P}({\cal U})$-valued process whose law is characterized as the unique solution of the following martingale problem. For any $\phi \in {\cal D}(A)$,
\be
M^x_{t}(\phi)
= \langle F^{u}_{t}(x,.),\phi\rangle -  \phi(u) - b(x) \int_{0}^t  \langle F^{u}_{s}(x,.),  A\phi\rangle ds\label{PBM-FV}
\ee
is a continuous square integrable  martingale with quadratic variation process
\begin{align}
\langle M^x(\phi)\rangle_{t}=  &\frac{b(x)+d(x)+\eta(x) C(0) \widehat{n}_x} {\widehat n_{x}}
\int_{0}^t  \left(\langle F^{u}_{s}(x,.), \phi^2\rangle -  \langle F^{u}_{s}(x,.), \phi\rangle^2 \right)ds\nonumber\\
= &\frac{2b(x)} {\widehat{n}_x}
\int_{0}^t  \left(\langle F^{u}_{s}(x,.), \phi^2\rangle -  \langle F^{u}_{s}(x,.), \phi\rangle^2 \right)ds.\label{crochet-PMB-FV}
\end{align}
\end{defi}


\bi
Let us now state our main theorem that describes the slow-fast dynamics of adaptive traits  and neutral markers at the (trait) evolutionary time scale.

\begin{thm}
  \label{thm:SFVP}
We work under Assumptions
  \ref{hypo} and \ref{IIF}. The initial conditions are\\
   $\nu^K_{0}(dy,dv)= n^K_{0}\,\delta_{(x_{0},u_{0})}(dy,dv)$ with
$\ \lim_{K\to \infty} n^K_{0} = \widehat n_{x_{0}}$ and $\sup_{K\in \N^*} \E((n^K_0)^3)<+\infty$.

\me
Then, the population process $\ (\nu^K_{Kt}, t\geq 0)$ converges in law to the $\mathcal{M}_{F}({\cal X}\times {\cal U}))$-valued process  $\,(V_t(dy,dv), t\geq 0)$ defined by
\be
\label{v}
V_{t}(dy, dv) =  \widehat{n}_{Y_{t}}\,\delta_{Y_{t}}(dy)\, F_{t}^{U_{t}}(Y_{t}, dv),
\ee
where the process $((Y_{t}, U_{t}), t\geq 0)$ on ${\cal X}\times {\cal U}$, started at $(x_0,u_0)$, jumps from $(x,u)$ to $(x+k,v)$ with the jump measure
  \begin{equation}
    \label{taux}
    b(x) \widehat n_{x}\frac{[f(x+k;x)]_+}
    {b(x+k)}\, F^{u}_t(x,dv)\,m(x,k)dk.
  \end{equation}

  \me
The convergence holds in the sense of finite dimensional distributions on $\mathcal{M}_{F}({\cal X}\times {\cal U})$.

\me  In addition, the convergence also holds in the sense of occupation measures, i.e. the measure $\nu^K_{Kt}(dy,dv) dt$ on $\X\times \U\times [0,T]$ converges weakly to the measure $\widehat{n}_{Y_{t}}\,\delta_{Y_{t}}(dy)\, F_{t}^{U_{t}}(Y_{t}, dv) dt$ for any $T>0$.\hfill $\Box$
\end{thm}

\begin{defi}
The limiting measure-valued process $( V_t(dy,dv), t\geq 0)$ is called Substitution Fleming-Viot Process. It generalizes the Trait Substitution Sequence (TSS) introduced by Metz et al. \cite{metzgeritzmeszenajacobsheerwaarden}.
\end{defi}

\noindent We observe that the Substitution Fleming-Viot Process includes the three qualitative behaviors due to the three different time scales: deterministic equilibrium for the transitory size of the population (driven by the ecological birth and death events), transitory diffusive behavior  for the marker distribution (driven by marker mutation), jump process  for the trait distribution (driven by adaptive trait mutation).

\bi
\begin{rem} Equations \eqref{PBM-FV}-\eqref{crochet-PMB-FV} have important biological implications regarding neutral genetic diversity. Once the fixation of a favorable mutation has occurred and the population is monomorphic for the selected trait, the evolution of the neutral marker distribution is described by a Fleming-Viot process whose law is given by the martingale in \eqref{PBM-FV}. The bracket of the martingale in  \eqref{crochet-PMB-FV} shows that the stochastic fluctuations with time of the marker distribution are due to randomness in births and deaths and mutations. The multiplicative factor $2b(x)/\widehat{n}_x$ in  \eqref{crochet-PMB-FV} depends on the trait value $x$, and on the assumed ecological model which determines the relationships between $x$, the death and birth rates and the competition kernel. Notice that $2b(x)/\widehat{n}_x$ corresponds to the quotient of variance (here $2b(x)$) divided by effective size $N_e$ (here $\widehat{n}_x$) that appears in the usual Wright Fisher equation. The quantity $\widehat{n}_x$ corresponds to the mass of the population when there is an infinite number of small individuals; if the size of the population is of order $K$, it means that there is approximately $n_x K$ individuals of weights $1/K$.
The right term in  \eqref{PBM-FV}, (i.e. the drift term in a mathematical sense) involves the generator $A$ and is associated with the mutation model as seen in Assumption \eqref{hyp-D}. The generator $A$ describes the speed at which the neutral diversity is restored. For instance in a continuous state model, if $A\phi=\frac{\sigma^2}{2}\phi''$, we recover the heat equation whose solutions have a variance in $t$. In a discrete state model similar to \eqref{A:bernoulli}, this equation gives the growth of the support.\\
In short, \eqref{PBM-FV}-\eqref{crochet-PMB-FV} shows that the distribution of the neutral marker depends on ecological processes and their parameters: every changes in $x$ will result in changes in the distribution of the neutral marker, through changes in birth, death and mutation rates, in competition and equilibrium population size. This result is biologically relevant and important since it differs from the assumptions of classical genetic hitchhiking models, in which selection and population size remain constant, leading to the fact that the neutral diversity restoration will not depend on the trait substitution and its history. In examples below, we will give more detailed results regarding about the distribution of the neutral marker changes.
\end{rem}

\bi
The proof of Theorem \ref{thm:SFVP} is the subject of Section \ref{section:proof}.

\me
The trait dynamics in the limit of Theorem \ref{thm:SFVP} is the Trait Substitution Sequence obtained in Champagnat \cite[Theorem 1]{champagnat06} whose assumptions are satisfied. Our main contribution in Theorem \ref{thm:SFVP} is to prove that at the adaptive trait mutation time scale, a homogeneization phenomenon takes place. There is a deterministic limit for the fastest process (the births and deaths leading to $\widehat{n}_x$), and stochastic limits for the two slower processes. The limiting process $(V_t, t\geq 0)$  is a measure-valued process with jumps (corresponding to trait mutations) and diffusion (corresponding to marker dynamics). If the population is trait-monomorphic with trait $x$, the jump rate is
$$    b(x) \widehat n_{x}\int_{\X-\{x\}} \frac{[f(x+k;x)]_+} {b(x+k)}\,m(x,k)dk.
$$When a jump occurs at $t$, the process jumps from $(x,u)$ to $(x+k,v)$ where $k$ is chosen in $m(x,k)dk$ and $v$ is chosen  at time $t$ in the marker distribution $F^u_t(x,dv)$.

\me
The marker distribution is the second fastest-evolving component, but marker mutations are assumed small \eqref{hyp-D}, allowing to recover a non-degenerate
 Fleming-Viot superprocess parameterized by the trait of the population but with jumps. Between the jumps, this superprocess is the pathwise limit of the marker dynamics where traits are fixed. The jumps are hitchhiking events due to the trait mutations (see in another context Etheridge, Pfaffelhuber and Wakolbinger \cite{etheridgepfaffelhuberwakolbinger}).  There is a bottleneck at each successful invasion-fixation of mutant traits. Indeed, the individuals present at the fixation time are all descendants of the successful initial mutant. The trait and marker of the latter alone determine the state of the new mutant population, hence creating the bottleneck for the whole population genealogy. This result is biologically intuitive since we assume that the neutral marker and the trait are completely linked, but the mathematical proof of these phenomena is the hardest part of the proof of Theorem \ref{thm:SFVP}, and we will show that our results still have biological interest. Extending this model to the case of recombination is a challenging problem for future work (see \cite{smadi} in this direction).\\
 It is also worth to notice that contrarily to other extensions of the TSS (e.g. the TSS with age-structure of \cite{meleardtran} or the Polymorphic Evolution Sequence for a multi-resource chemostat in \cite{champagnatjabinmeleard}) that usually jump from an equilibrium to another equilibrium, the marker distribution is here described by a stochastic process and not an equilibrium measure. This is due to the fact that the time scales of the trait and marker mutations are assumed different: in the time scale of marker mutations, the trait mutations are too rare and not seen.

An illustration of the invasion and fixation phenomena is summed up in Fig. \ref{fig:inv-fix}.

  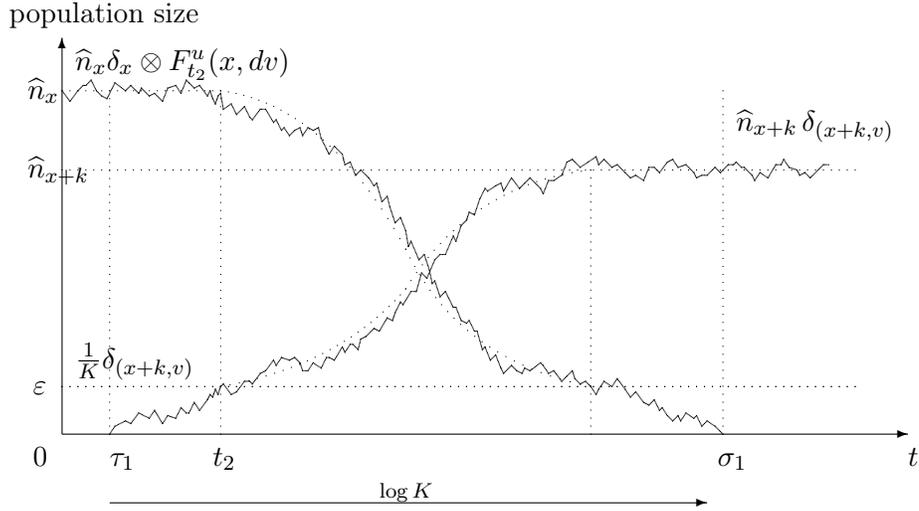
\begin{figure}[ht]
  \begin{center}
    \begin{picture}(350,180)(-20,-10)
      \put(0,0){\vector(1,0){320}} \put(0,0){\vector(0,1){150}}
      \put(-11,-12){0}
      \put(-11,15){$\varepsilon$}
      \put(-13,97){$\widehat{n}_{x+k}$}
      \put(-13,127){$\widehat{n}_x$}
      \put(5,138){$\widehat{n}_x \delta_x \otimes F^u_{t_2}(x,dv)$}
      \put(5,25){$\frac{1}{K}\delta_{(x+k,v)}$}
      \put(-20,156){population size}
      \put(57,-12){$t_2$}
      \put(18,-26){\vector(1,0){226}}
      \put(120,-24){{\scriptsize $\log K$}}
      \put(248,-12){$\sigma_1$}
      \put(320,-12){$t$}
      \put(18,-12){$\tau_1$}
      \put(255,115){$\widehat{n}_{x+k}\, \delta_{(x+k,v)}$}
      \dottedline{3}(0,18)(300,18)
      \dottedline{3}(18,0)(18,130)
      \dottedline{3}(0,100)(300,100)
      \dottedline{3}(0,130)(60,130)
      \dottedline{3}(60,0)(60,130) \dottedline{3}(200,0)(200,100)
      \dottedline{3}(0,18)(300,18)
      \dottedline{3}(250,0)(250,130)
      \qbezier[27](60,130)(100,125)(130,70)
      \qbezier[25](130,70)(150,28)(200,18)
      \qbezier[25](60,18)(105,26)(130,55)
      \qbezier[25](130,55)(155,95)(200,100)
      \dottedline{0.5}(18,0)%
(20,3)(24,5)(26,4)(29,8)(33,6)(35,9)(38,4)(40,7)(43,6)(45,10)(47,8)(48,8)%
(50,12)(52,9)(55,13)(56,15)(57,14)(58,17)(59,15)(60,18)(61,19)(64,15)%
(66,16)(69,20)(71,20)(72,19)(75,23)(77,24)(80,28)(81,26)(83,29)(86,28)%
(88,29)(90,27)(91,24)(93,24)(95,27)(98,25)(100,29)(101,29)(103,28)(106,33)%
(107,31)(109,34)(110,34)(112,32)(113,36)(115,37)(118,40)(120,39)(123,44)%
(125,43)(126,46)(128,50)(129,49)(132,54)(134,54)(136,61)(138,59)(141,66)%
(143,68)(145,68)(148,75)(149,73)(151,79)(153,81)(154,84)(156,83)(158,89)%
(160,93)(161,92)(164,94)(166,95)(167,92)(170,97)(173,93)(175,95)(177,97)%
(181,93)(182,91)(184,96)(186,96)(189,98)(191,103)(193,104)(196,101)(199,103)%
(202,105)(203,102)(205,104)(208,101)(210,102)(213,99)(215,97)(217,98)%
(220,96)(222,101)(224,101)(226,99)(229,103)(231,104)(233,101)(236,102)%
(239,99)(240,97)(242,98)(244,98)(247,101)(249,99)(252,102)(254,102)(255,103)%
(257,101)(258,101)(261,98)(262,96)(264,99)(267,97)(270,102)(272,102)%
(273,104)(275,101)(278,100)(279,98)(281,99)(283,97)(288,102)(290,102)
      \dottedline{0.5}(0,130)(3,126)(6,130)(8,132)(9,132)(11,134)(13,130)%
(15,128)(17,127)(20,133)(24,130)(26,132)(29,129)(30,131)(33,128)(35,129)%
(38,126)(40,131)(43,132)(45,129)(47,134)(48,133)(50,131)(52,132)(55,128)%
(56,130)(57,127)(58,125)(59,128)(60,129)(61,123)(64,125)(66,121)(69,124)%
(71,120)(72,119)(75,123)(77,123)(80,118)(81,119)(83,115)(84,116)(86,113)%
(88,115)(90,116)(91,113)(93,116)(95,115)(97,116)(98,112)(100,108)(101,106)%
(103,109)(106,106)(107,102)(109,103)(111,98)(113,100)(116,97)(118,94)%
(119,95)(121,88)(123,90)(124,86)(126,80)(128,82)(130,77)(131,71)(132,73)%
(135,66)(137,68)(139,62)(140,57)(141,58)(143,52)(145,54)(148,48)(149,48)%
(151,43)(153,42)(154,43)(156,39)(157,39)(159,32)(160,33)(162,28)(164,30)%
(166,25)(167,23)(170,26)(172,24)(174,23)(175,25)(177,24)(181,27)(182,26)%
(184,23)(186,24)(189,20)(191,22)(193,24)(194,21)(196,19)(198,20)(200,18)%
(202,15)(203,17)(205,20)(208,18)(210,21)(213,17)(215,14)(217,16)(220,12)%
(222,11)(224,13)(226,9)(227,10)(229,7)(231,5)(233,7)(234,5)(236,8)%
(239,6)(240,7)(242,3)(244,5)(247,3)(249,1)(250,0)
    \end{picture}
  \end{center}
  \caption{{\small Invasion and fixation of a successful trait mutant. In the population of resident trait $x$ and marker distribution $F_t^u(x,dv)$, a mutant trait $x+k$ appears at time $\tau_1$. Let $v$ be the marker of the mutant individual. As in Champagnat et al. \cite{champagnatferrieremeleard2}, the fluctuations of the resident population can be neglected in first approximation and the mutant population evolves as a birth and death process with rates $b(x+k)$ and $d(x+k)+\eta(x+k)C(k)\widehat{n}_x$, independent of the marker distribution. When the mutant population reaches a sufficient size $\varepsilon$ at time $t_2$, with probability $[f(x+k,x)]_+/b(x+k)$, the `invasion implies fixation' assumption leads to the replacement of the former population in a time $t_K$ such that $t_K/\log(K)\rightarrow\infty$. This time interval is too short to allow other marker mutant to appear in non-negligible proportion, with large probability. Thus, when the mutant population has fixed, at time $\sigma_1$, it is close to $\widehat{n}_{x+k} \delta_{(x+k,v)}$. Before the next adaptive trait mutation occurs, the marker mutates a lot, since marker mutations happen on a faster scale. The dynamics of the marker distribution is then the one of a Fleming-Viot superprocess started at $\delta_v$ and with statistics depending on $x+k$.}}
  \label{fig:inv-fix}
\end{figure}

\subsection{An example from Dieckmann and Doebeli}\label{section:DD}

Let us first illustrate our model by simulations based on an example inspired from Roughgarden \cite{roughgarden} and  Dieckmann and Doebeli \cite{dieckmanndoebeli}. Here $\X=[-1,1]$, $\U=[-2,2]$ and $K=1000$ in all the simulations. The individual dynamics is characterized by
\begin{itemize}
   \item the birth rate $b(x)=\exp(-x^2/2\sigma_b^2)$ with $\sigma_b = 0.9$. The probability of mutation of the trait and  marker are respectively $p_{K}=1/K^2$ and $q_{K}=1/\sqrt{K}$. The adaptive trait mutation kernel $m(x,k) dk$ is a Gaussian law with mean 0 and variance 0.1, conditioned to $[-1,1]$. The marker mutation kernel  $G_K(u,dh)$  is a Gaussian law with mean 0 and variance $\sigma^2_K=1/\sqrt{K}$, conditioned to $[-2,2]$.
   \item symmetric competition for resources, with $ \eta(x)=1$ and
   $C(x-y)=\exp(-(x-y)^2/2\sigma^2_C)$, $\sigma_C = 0.8$.
  \end{itemize}
Here, the `optimal trait' is $x=0$ where the birth rate has its maximum and the population is governed by local competition.
We start with the initial condition: $\ x_{0}= -1, u_{0}= 0$.\\

\begin{figure}[!ht]\label{fig:WF}
\begin{center}
\begin{tabular}{cc}
  (a) & (b) \\
& \\
\hspace{0cm}\includegraphics[width=4cm,angle=0,trim= 6cm 6cm 6cm 6cm]{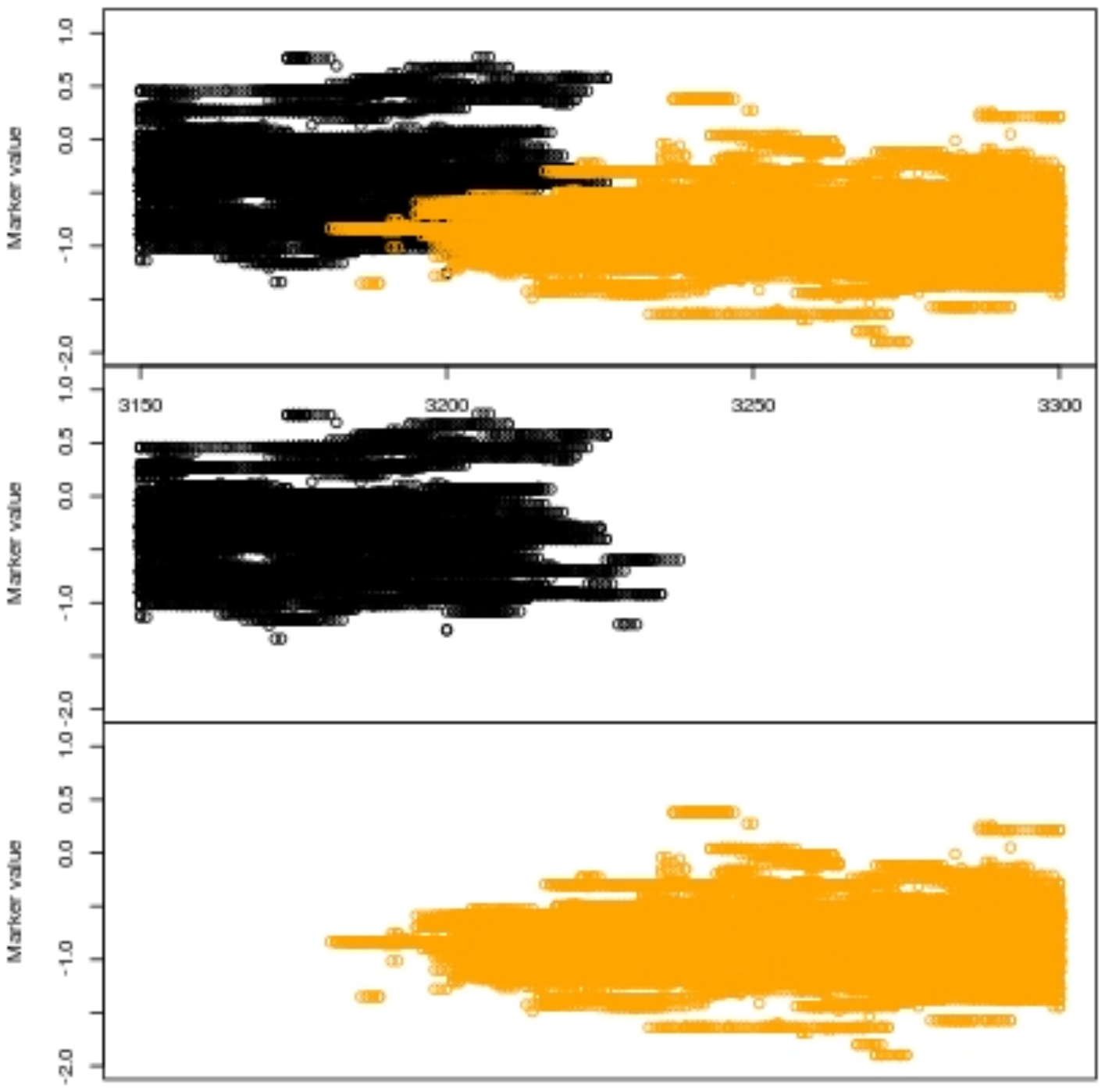}
&
\hspace{2cm} \includegraphics[width=4cm,angle=270,trim= 24.5cm 6cm -16.5cm 4cm]{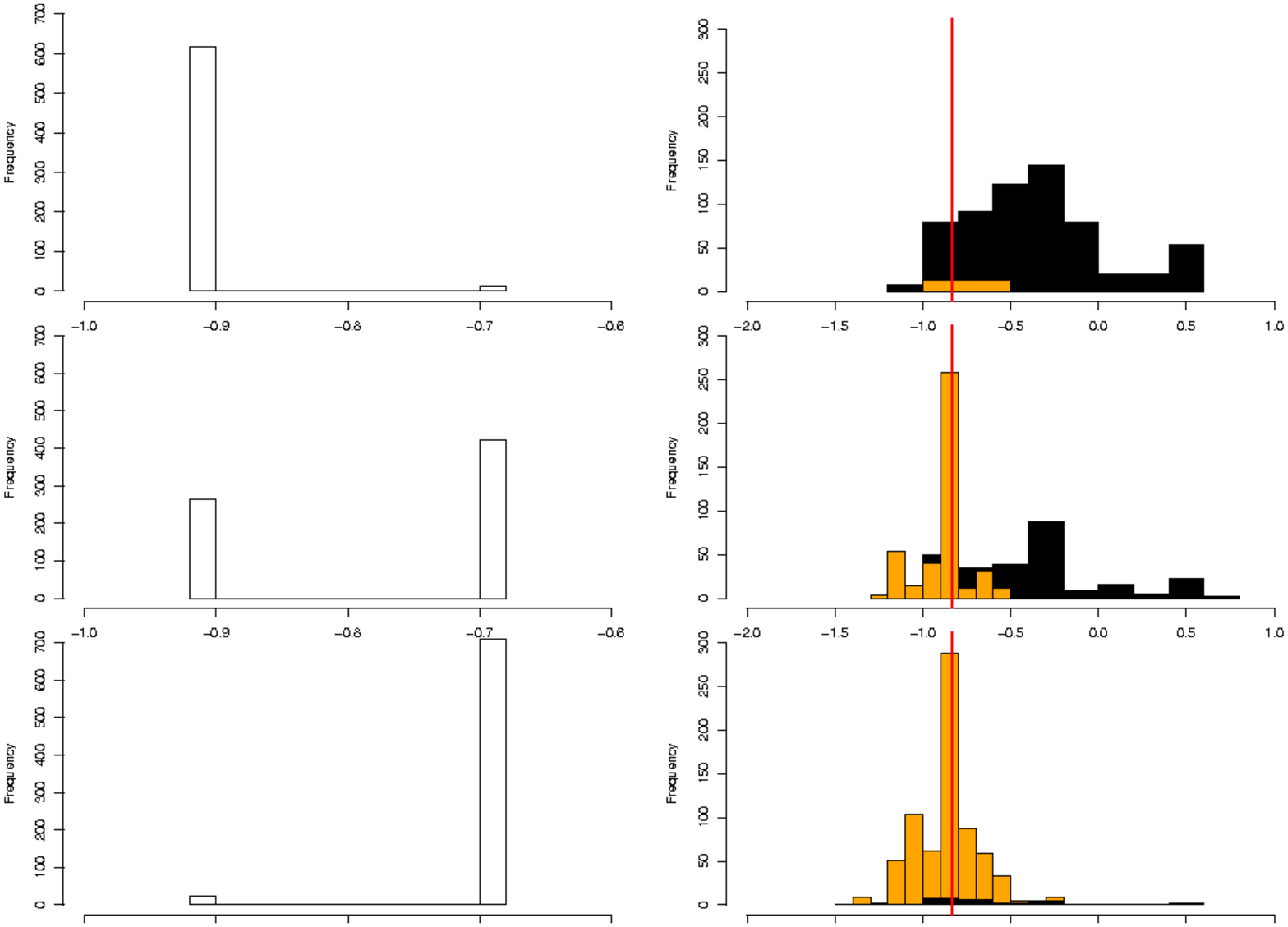}
\end{tabular}
\vspace{-5cm}
\caption{\textit{We consider a resident trait-monomorphic population (black) in which a mutant trait (light) appears and goes to fixation. Here the intrinsic death rate $d(x)=0$. (a) Evolution of the support of the marker distribution with respect to time; the support of the resident trait-monomorphic population is in black while the support of the mutant population is lighter. The mutant and resident populations are shown separately and together. (b) Distributions of the traits (left) and markers (right) in the population at three times during the invasion. The marker and trait values are in abscissa and frequency is in ordinate. The marker value of the initial mutant is indicated by the red line. When the mutant trait appears, the resident population is quickly invaded by the mutant population during a transition period. In (b), we can see that if the support of the marker distribution for the resident population remains wide (see also (a)), the size of the resident population decreases quickly. In the second column of (b), we see that the marker distribution in the mutant population remains spiked at the marker value of the first mutant individual during the whole transition period. After invasion (see (a)), the spread of the marker distribution follows the Fleming-Viot process \eqref{PBM-FV}-\eqref{crochet-PMB-FV}. On (a), we see that for the Fleming-Viot process, the support of the marker distribution spreads slowly.}}\label{simu1}
\end{center}
\end{figure}

\noindent The simulations (see Fig. \ref{simu1}) illustrate Theorem \ref{thm:SFVP}. They show the replacement of a resident population by a mutant population.
In Fig. \ref{simu1} (a), the dynamics of the support of the marker distribution is represented. The mutant and resident populations are pictured together and separately to better observe the extinction of the resident population (black) and the expansion of the mutant population from one individual (light). The invasion started around time 3175 is quick and after time 3250, the mutant population has totally replaced the resident one.\\
In Fig. \ref{simu1} (b), the histograms of traits and markers at three times during the invasion are represented simultaneously, to underline the hitchhiking effect of the marker during the `invasion implies fixation' phase. We can see that the distribution of the marker during the fixation remains close to a Dirac mass at the marker value of the initial mutant (red line). This illustrates the bottleneck phenomenon, the existence of which we prove rigorously (Equation \eqref{distrib-marker} of Proposition \ref{prop:hitchhiking}).



\me
Let now focus on the Dieckmann-Doebeli's example and highlight the biological implications regarding the eco-evolutionary feedback on the distribution of the neutral marker. Here, $\widehat{n}_x=b(x)-d(x)$ and therefore the Fleming-Viot process $F^u_t(x,.)$ is the solution of the martingale problem \eqref{PBM-FV} with $A\phi=\frac{\sigma^2}{2}\phi''$ and with bracket \eqref{crochet-PMB-FV} given for all $\phi\in \Co(\U,\R)$ by $\
2\frac{b(x)}{b(x)-d(x)}\int_0^t \big(\langle F^u_s(x,.),\phi^2\rangle -\langle F^u_s(x,.),\phi\rangle^2\big)ds.$ \\
If the death rate is a constant $d(x)=d$, then the multiplicative factor in the bracket \eqref{crochet-PMB-FV}, $b(x)/(b(x)-d)$, decreases when $b(x)$ increases. 
Heuristically we expect that the stochastic fluctuations in time of the distribution of the neutral marker decrease when the trait $x$ approaches the evolutionary stable strategy (ESS, see \cite{maynardsmith}) and $b(x)$ increases, since the equilibrium size is greater and the diffusion coefficient is lower. The drift term is $b(x)\frac{\sigma^2}{2}\int_0^t \langle F^u_s(x,.),\phi''\rangle ds$ and thus the multiplicative factor $b(x)$ increases when approaching the ESS, contrarily to the multiplicative factor of the bracket \eqref{crochet-PMB-FV}. In the case $d\equiv 0$, the Fleming-Viot process has a constant diffusion coefficient and the bracket \eqref{crochet-PMB-FV} does not depend on $x$. The Fleming-Viot process depends only on the trait $x$ through the drift term. Notice that this is true for any mutation model satisfying \eqref{hyp-D}. This simple result illustrates how the ecological processes can shape the neutral diversity.


\subsection{Corollary: the Wright-Fisher Evolutionary Process}\label{section:WF}

There exists a version of the SFVP in the case when the marker space $\U$ is discrete. Assume for instance that there exist only two alleles of the marker trait, denoted by $a$ and $A$, so that $\U=\{a,A\}$. In this case, we apply Theorem \ref{thm:SFVP} with the mutation kernel $G_K$ defined in \eqref{DK:bernoulli} and $r_K/K\rightarrow \bar{r}>0$ when $K\rightarrow +\infty$.

\begin{prop}\label{prop:SFVP-WF}
We work under Assumptions \ref{hypo} and \ref{IIF} with probabilities $q_A$ and $q_a$ to mutate from marker $A$ to marker $a$ and from marker $a$ to marker $A$. Moreover, we consider similar initial conditions $\nu^K_{0}$ as in Theorem \ref{thm:SFVP}.
Then, the population process $\ (\nu^K_{Kt}, t\geq 0)$ converges in law to the $\mathcal{M}_{F}({\cal X}\times \{a,A\})$-valued process
$$( \widehat{n}_{Y_{t}}\,\big(W_t^a\  \delta_{(Y_{t},a)}(dy,du)+(1-W_t^a)\ \delta_{(Y_t,A)}(dy,du)\big), t\geq 0),$$where  $(Y_{t}, t\geq 0)$ is the TSS process that jumps from $x$ to $x+k$ in $\X$ with the jump measure $\,b(x)\, \widehat n_{x}\,\frac{[f(x+k;x)]_+}{b(x+k)}\,m(x,k)dk\ $ and where $(W_t^a, t\geq 0)$ is the following Wright-Fisher jump process that represents the proportion of alleles $a$ in the population of trait $Y_t$ at time $t$. Between jumps, it satisfies the usual Wright-Fisher equation with mutations
    \begin{equation}\label{WF}
    dW^a_t= \bar{r}\, b(Y_t) \big( q_A(1-W^a_t) - q_a W^a_t \big) dt + \sqrt{\frac{2 b(Y_t)}{\widehat{n}_{Y_t}}\ W^a_t\ \big(1-W^a_t\big)} dB_t
    \end{equation}$(B_t, t\geq 0)$ being a standard Brownian motion. It jumps with the TSS and at jump time $t$, the process $(W^a_t,1-W^a_t)$ goes to $(1,0)$ with probability $W^a_t$ and to $(0,1)$ with probability $1-W^a_t$.
\hfill $\Box$
\end{prop}

\me
An illustration of this theorem is given in Fig. \ref{simu-WF}.\\

This result can be generalized to discrete marker spaces $\U=\{a_1,\dots a_m\}$, by introducing the transition probabilities $q_{ij}$ to mutate from $a_i$ to $a_j$, $i,j\in \{1,\dots,m\}$. An application is when the marker corresponds to the genetical sequence of $n$ nucleotides ($A$, $T$, $G$ or $C$ for each position). In this case, $m=\card \ \U=4^n$.

\me Traditionally in a population genetics framework, the evolution in finite populations of the diversity at a neutral marker is described as a diffusion process with two fixed parameters: the population size and the mutation ``rate'' (e.g. Crow and Kimura 1970). The population size is related to what is called the ``genetic drift'' and generally refers to the random sampling of gametes performed for reproduction at the beginning of each generation, and the higher the population size, the lower the genetic drift. Under this framework, genetic drift induces stochastic fluctuations in the frequencies of the alleles $A$ and $a$ and can cause the decrease of neutral genetic diversity when an allele is randomly lost. On the other hand, mutation introduces continuously allele $A$ and $a$ in the population and thus allows the restoration and the maintenance of neutral genetic diversity. It is important to note that under the population genetics framework, mutation rates and population size are fixed and do not depend on the ecological processes and their parameters, neither on the trait value when the population is monomorphic for the trait under selection. As a consequence, those parameters do not change as successive selective sweeps occur especially during the adaptation process. Here we can use Equation \eqref{WF} and try to compare the classical population genetics results about the distribution of neutral diversity and the one in our model.

\noindent In an eco-evolutionary framework,  \eqref{WF} first shows that mutation rates and population size, i.e. the genetic drift, are not fixed and depend on the ecological processes and on the trait value $x$. The mutation rates are $\bar{r} \,b(Y_t) q_A $ and $\bar{r}\, b(Y_t) q_a$ in our framework while it is only $q_A$ and $q_a$ under a population genetics framework (e.g. Crow and Kimura 1970). The genetic drift, i.e. the equilibrium population size, is given by $1/\widehat{n}_{Y_t}$ while it is a constant $1/n\,$ in population genetics framework. Second,  \eqref{WF} shows that extra ecological processes affect the distribution of the neutral marker since in the left-hand side there is the term $2 b(Y_t)$. This term can be interpreted as the effect of demographic stochasticity, which is not taken into account in population genetics.

\begin{figure}[!ht]
\begin{center}
\hspace{0cm}\includegraphics[width=6cm,angle=270,trim= 1cm 1cm 1cm 1cm]{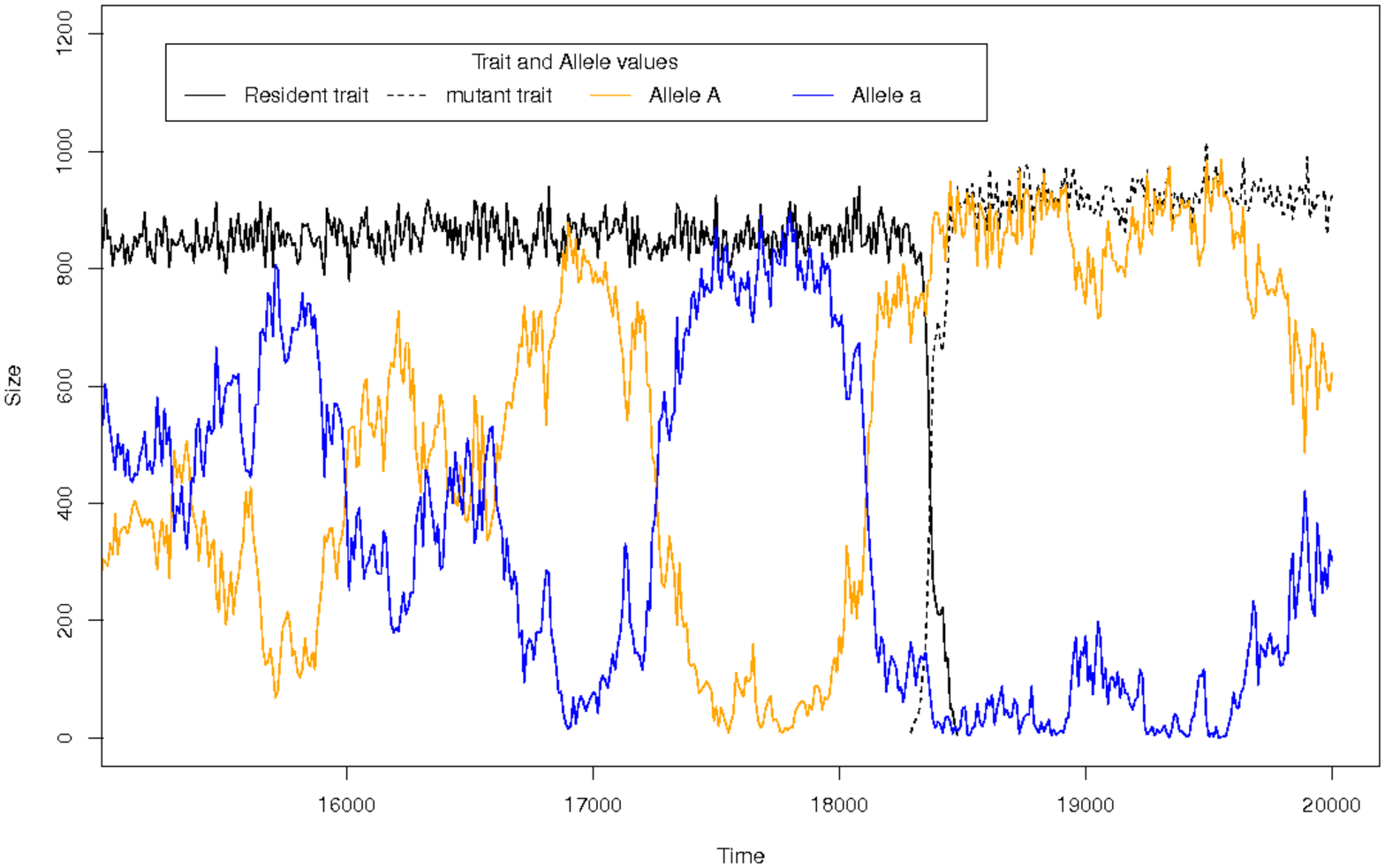}
\caption{\textit{Evolution of sizes of the subpopulations with markers $a$ and $A$. The simulation uses individual-based algorithms. The proportions of marker alleles $a$ and $A$ follow Wright-Fisher diffusions while the size of the population stabilizes around the equilibrium given by the trait value. A trait mutant appears around time 18290, invades and fixes into the population. Before the appearance of this mutant trait, fluctuations in the marker distribution are due to (fast) marker mutation, birth and death stochastic events. At the time when the mutant trait appears, the $A$-allele frequency is 85\%, giving a high probability for an A-allele hitchhike. This is the case in the simulation. After fixation time (around time 18490), the $a$-allele population is extinct. It is regenerated by mutations of the marker but get extinct three times before taking up around time 19600.
}}\label{simu-WF}
\end{center}
\end{figure}

\subsection{Extensions to co-existing traits}\label{section:coexistence}



The work of Champagnat and Méléard \cite{CM2011} generalizes the TSS to the case of coexisting trait values, when Assumption \ref{IIF} is relaxed. They define a polymorphic TSS called polymorphic evolutionary sequence (PES) and denoted by $(X_t)_{t\geq 0}\in \D(\R_+,\mathcal{M}_F(\X))$.
When a mutant trait $y$ appears in a resident population of trait $x_0$ at time $t_1$, either its descendent line is killed with probability $1-[f(y;x_0)/b(y)]_+$, or it survives. In that case, we can have coexistence of $y$ and $x_0$ when there is a positive globally stable non-trivial equilibrium $(n^*_{x_0,y}, n^*_{y,x_0})$ to the Lotka-Volterra system defined in \eqref{largepop-dimorphique}. Therefore the population jumps from $X_{t_1-}=\widehat{n}_{x_0}\delta_{x_0}$ to
$$X_{t_1}=n^*_{x_0,y} \delta_{x_0}(dx)+ n^*_{y,x_0} \delta_{y}(dx). $$

For a probability $\pi$, a trait measure $X$ and $x\in \X$, let us denote by $F_t(\pi,x,X, du)$ the Fleming-Viot process started at $\pi$, evolving in the trait distribution $X$ and parameterized by $x$.

Let $\pi_0$ be the initial marker distribution of the monomorphic population of trait $x_0$. Before the time $t_1$ of appearance of the first mutant, the marker distribution evolves as $(F_t(\pi_0,x_0,\widehat{n}_{x_0}\delta_{x_0},du))_{t\geq 0}$. Let $\pi_{t_1}=F_{t_1}(\pi_0,x_0,\widehat{n}_{x_0}\delta_{x_0},du)$ be the marker distribution at $t_1$ and let $V_1$ be a random variable drawn in the distribution $\pi_{t_1}$. After $t_1$ and before the occurence of the second trait-mutation at $t_2$, the population evolves as
$$n^*_{x_0,y} \delta_{x_0}(dx) F_{t-t_1}(\pi_{t_1},x_0,X_{t_1},du)+n^*_{y,x_0} \delta_y(dx)F_{t-t_1}(\delta_{V_1},y,X_{t_1},du).$$
The processes $F_{t}(\pi_{t_1},x_0,X_{t_1},du)$ and $F_{t}(\delta_{V_1},y,X_{t_1},du)$ are generalizations of the Fleming-Viot process defined in Definition \ref{def:FV}. Indeed their semimartingale decompositions are respectively:

\be \langle F_{t}(\pi_{t_1},x_0,X_{t_1},.),\phi \rangle& = & \langle \pi_{t_1},\phi \rangle +  b(x_{0}) \int_{0}^t \langle  F_{s}(\pi_{t_1},x_0,X_{t_1},.), A \phi \rangle \ ds + M^1_{t}(\phi)\ ; \nonumber \\
\langle F_{t}(\delta_{V_1},y,X_{t_1},.),\phi \rangle& = &  \phi(V_1) +  b(y) \int_{0}^t \langle  F_{s}(\delta_{V_1},y,X_{t_1},.), A \phi \rangle\ ds + M^2_{t}(\phi),
\ee where $M^1(\phi)$ and $M^2(\phi)$ are independent square integrable martingales such that
\begin{align}
\langle M^1(\phi)\rangle_{t}=  &\frac{b(x_{0})+d(x_{0})+\eta(x_{0}) C(0) n^*_{x_0,y}+ C(x_{0}-y)n^*_{y,x_0}} {n^*_{x_0,y}+n^*_{y,x_0}}\nonumber\\
& \hskip 2cm
\int_{0}^t  \left(\langle F_{s}(\pi_{t_1},x_0,X_{t_1},.), \phi^2\rangle - \langle F_{s}(\pi_{t_1},x_0,X_{t_1},.), \phi\rangle^2 \right)ds,\nonumber\\
\langle M^2(\phi)\rangle_{t}=  &\frac{b(y)+d(y)+\eta(y) C(y-x_{0}) n^*_{x_0,y}+ C(0)n^*_{y,x_0}} {n^*_{x_0,y}+n^*_{y,x_0}}\nonumber\\
& \hskip 2cm
\int_{0}^t  \left(\langle F_{s}(\delta_{V_1},y,X_{t_1},.), \phi^2\rangle -  \langle F_{s}(\delta_{V_1},y,X_{t_1},.), \phi\rangle^2 \right)ds.
\end{align}
At time $t_2$, when a third trait appears in the population, the system can evolve to three two or just one coexisting traits, depending on the new trait equilibrium of the Lotka equations that is reached. For each of the traits, the marker distribution evolves as a generalization of the Fleming-Viot processes above.

\bi
\begin{rem} The above equations show that, when there is coexistence of two traits in the population, the markers in the  subpopulations defined by the two traits  evolve independently but with parameters depending on the two co-existing traits.  Thus, when there is a diversification event in the population, the distribution of the neutral diversity in one of the two subpopulations does not evolve  as  completely forgetting  the other one, as it is usually assumed. The parameters of the underlying Fleming-Viot process depend on the complete trait distribution.
\end{rem}

\bi We present in Figure  \ref{simu2}  simulations in the case of coexistence, with the same model and parameters as in Section \ref{section:DD}, except $\sigma_C = 0.7$ and the initial condition: $\ x_{0}= -0.1.$ The simulations (see Fig. \ref{simu2}) show the appearance of a new mutant trait (yellow) in a population of two coexisting traits (black and blue).


\begin{figure}[!ht]
\begin{center}
\begin{tabular}{cc}
  (a) & (b) \\
& \\
\hspace{0cm}\includegraphics[width=4cm,angle=0,trim= 6cm 6cm 6cm 6cm]{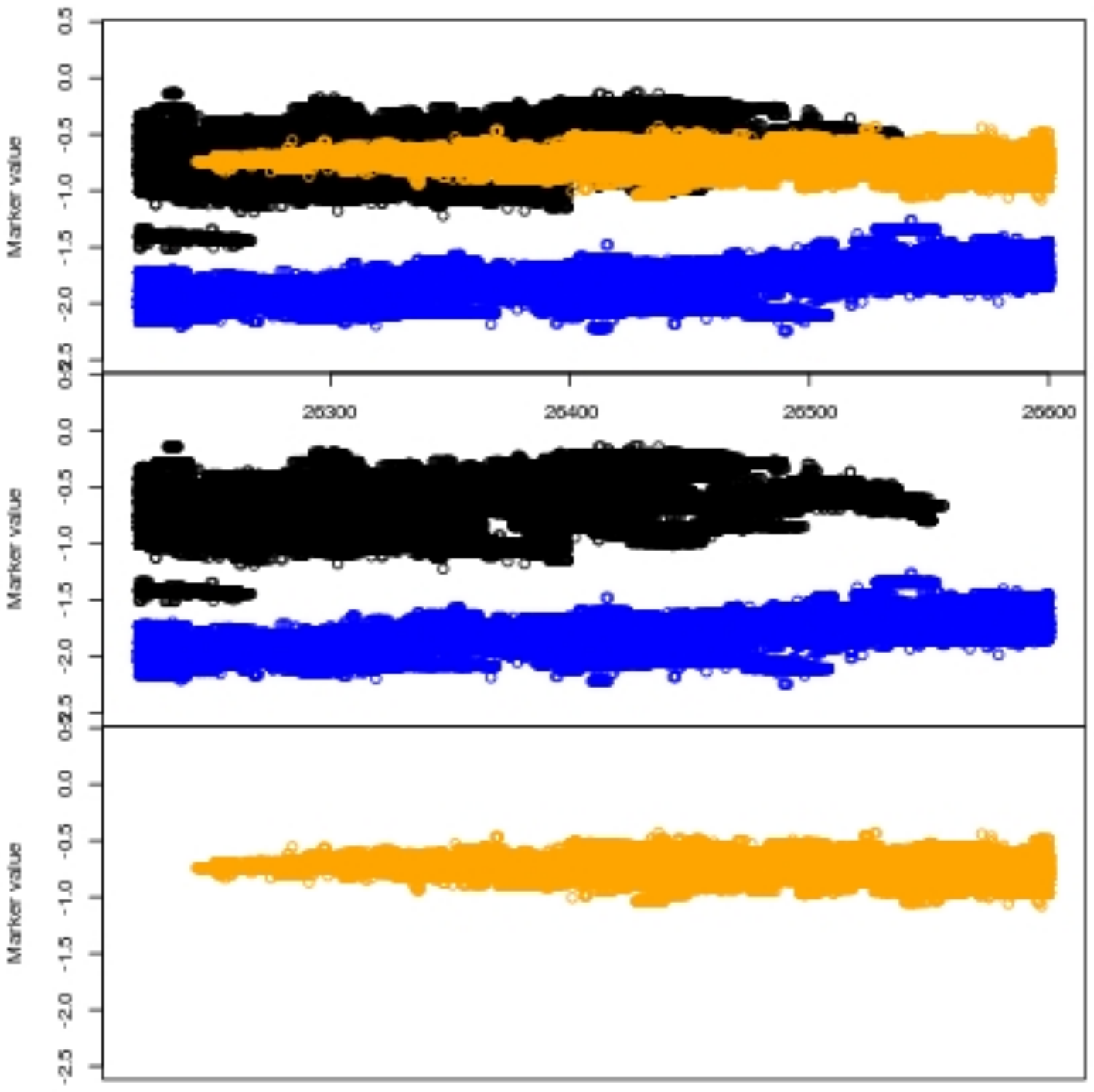}
&
\hspace{2cm} \includegraphics[width=4cm,angle=270,trim= 24.5cm 6cm -16.5cm 4cm]{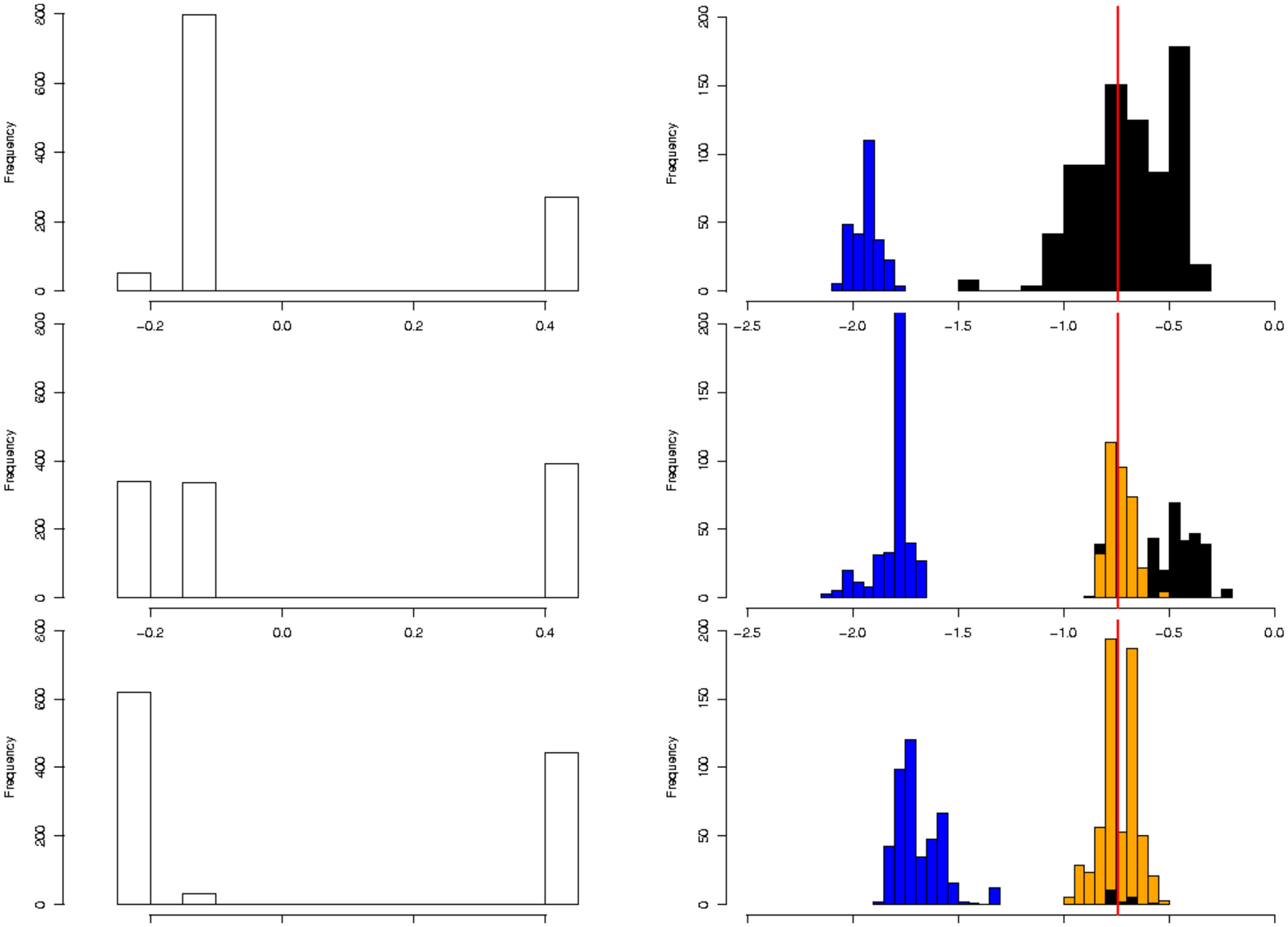}
\end{tabular}
\vspace{-5cm}
\caption{\textit{Neutral marker dynamics in a trait-dimorphic population. Evolution of markers with respect to time. (a) The mutant population (yellow) and resident populations (black and blue) are shown separately and together. (b) Distributions of the traits and markers in the population at three times during the invasion. The marker value of the initial mutant is indicated by the red line. }}\label{simu2}
\end{center}
\end{figure}


\section{Proof of Theorem \ref{thm:SFVP}}\label{section:proof}



Let us sketch the proof. In this section, we will suppose that Assumptions \ref{hypo}, \ref{IIF} are satisfied and the initial conditions are $\nu^K_{0}(dy,dv)= n^K_{0}\,\delta_{(x_{0},u_{0})}(dy,dv)$ with
$\ \lim_{K\to \infty} n^K_{0} = \widehat n_{x_{0}}$ and $\sup_{K\in \N^*} \E((n^K_0)^3)<+\infty$.

\me First, we recall results due to Champagnat et al. \cite{champagnatferrieremeleard2} that provide the finite marginal convergence of the trait process $({X}^K_{Kt} ; t\geq 0)$. We extend these results to obtain the weak convergence of the measures $(X^K_{Kt}(dx)dt ; K\geq 0)$ in $\mathcal{M}_F(\X\times[0,T])$ embedded with the weak convergence topology. This corresponds to the convergence of $({X}^K_{Kt} ; t\geq 0)$ in the sense of occupation measures, as developed by Kurtz \cite{kurtzaveraging}. Secondly, we include the fast component (the marker) and prove the tightness of the sequence $(\nu^K_{Kt}(dx,du)dt ; K\geq 0)$ in $\mathcal{M}_F(\X\times \U\times [0,T])$. We then consider a subsequence, again denoted by $(\nu^K_{Kt}(dx,du)dt, K\geq 0)$ with an abuse of notation, that converges to a limit $\Gamma(dt,dx,du)\in \mathcal{M}_F([0,T]\times \X\times \U)$ that we have to identify. This derivation is done in several steps. When a successful mutant appears in the monomorphic population with trait $x$, the transition period to fixation is to be considered carefully. It has been proved in \cite{champagnat06} that these transitions are of order $\log(K)$. We prove that during this time interval, the marker distribution in the mutant subpopulation remains a Dirac mass at the value of the initial mutant. This results from the combined effects of small or rare marker mutations, large population and slow take-off of the new mutant population. Then, we show that in a trait monomorphic population with value $x$, the marker distribution converges to a Fleming-Viot superprocess parameterized by $x$.

\subsection{Semimartingale decomposition of $\nu^K$}

Let us introduce some notation to keep forthcoming formula simple. For $\nu\in \mathcal{M}_F(\X\times \mathcal{U})$ and $\phi(x,u)\in \Co(\mathcal{X}\times \mathcal{U},\R)$, we define the (nonlinear) generators $B^K$ and $D^K(\nu)$ such that
\begin{align}
B^K\phi(x,u)= & (1-p_K)(1-q_K) b(x)\phi(x,u)\nonumber\\
 +&   p_K(1-q_K)b(x) \int_{\mathcal{X}} \phi(x+k,u) m(x,k)dk \nonumber\\
+ &  q_K(1-p_K)b(x) \int_{\mathcal{U}} \phi(x,u+h) G_K(u,dh) \nonumber\\
+ &  p_K\ q_K \ b(x) \int_{\mathcal{X}\times \mathcal{U}} \phi(x+k,u+h) m(x,k) dk\ G_K(u,dh) \label{def:B(nu)}\\
D^K(\nu)\phi(x,u)= & \big(d(x)+\eta (x) C* \nu(x)\big)\phi(x,u).\label{def:D(nu)}
\end{align}

\bi
The process $\langle \nu^K_.,\phi\rangle$ is a  square integrable semi-martingale and we give its characteristics.


\begin{prop}\label{prop:nuK}
For a continuous bounded function $\phi(x,u)$ on $\mathcal{X}\times \mathcal{U}$, the process
\begin{align}
M^{K,\phi}_t 
= & \langle \nu^K_t,\phi\rangle - \langle \nu^K_0,\phi\rangle - \int_0^t ds  \int_{\mathcal{X}\times \mathcal{U}} \nu^K_s(dx,du) \big(B^K-D^K(X^K_s)\big)\phi(x,u)\label{MKf}
\end{align}is a square integrable martingale with previsible quadratic variation
\begin{align}
\langle M^{K,\phi}\rangle_t
= & \frac{1}{K}\int_0^t ds  \int_{\mathcal{X}\times \mathcal{U}} \nu^K_s(dx,du) \big(B^K+D^K(X^K_s)\big)\phi^2(x,u).\label{crochetMKf}
\end{align}
\end{prop}

\begin{proof}The dynamics being given in Section \ref{sec:IBM}, the proof can be adapted from  Fournier and Méléard \cite[Lemma 5.2]{fourniermeleard}. One main step consists in showing that there exists a Poisson point measure driving the measure-valued processes $\nu^K$ for all $K\in \N^*$.
\end{proof}

\subsection{Convergence of the trait-marginal in the trait mutation time scale}

\bi\noindent As previously emphasized, the trait dynamics is described by the measure-valued process $ X^K$ which does not depend on the markers. This process  has been fully studied in \cite{champagnat06,champagnatferrieremeleard2}.
In this section, we recall the finite marginal convergence result obtained in these papers. We  give some additional properties concerning the topology involved. This result shows a time scale separation with successive fixations of successful mutants, under Assumptions \ref{hypo} and \ref{IIF}. Notice that the time scale assumption is
  \begin{equation}
    \label{eq:u_K-K}
    \forall V>0,\quad \log K\ll \frac{1}{K p_K}\ll \exp(VK), \quad \hbox{ as } K\to \infty,
  \end{equation}which is realized in our case for $p_K= 1/K^2$.

   \me
\begin{thm}  \label{thm:TSS}
  Under Assumptions \ref{hypo} and \ref{IIF}, let us also assume that the initial population is trait-monomorphic: $ X^K_0=n^K_0\delta_{x}$ for $x\in {\cal X}$ and   $n^K_0\rightarrow \widehat{n}_{x}$ in probability  and $\sup_{K\in \N^*} \E((n^K_0)^3)<+\infty$.

 \me
  Then, the sequence $( X^K_{Kt}; t\geq 0)$ converges to the pure jumps singleton measure-valued Markov process
  $(\widehat{n}_{Y_{t}}\,\delta_{Y_{t}}; t\geq 0)$ defined as follows: $Y_0=x$,
  and the process $Y$ jumps from
  $\,{x}$ to $\ {x+k}$
  with jump measure
  $  \  b(x)\, \widehat n_{x}\frac{[f(x+k;x)]_+}{b(x+k)}\,m(x,k)dk.
 $

 \me
  The convergence holds in the sense of finite dimensional distributions on ${\cal M}_F(\X)$ equipped with the topology of total variation. \end{thm}

\me This theorem has been proved in Champagnat \cite{champagnat06}
 for the logistic case and generalized in \cite{champagnatferrieremeleard2}.

 \me The trait-marginal process $(X^K_{Kt} ; t\in [0,T])$ does not converge in $\D([0,T],\mathcal{M}_F(\X))$ embedded with the Skorokhod topology. Indeed, the size of jumps is upperbounded by ${1\over K}$ and nevertheless the limiting  total mass process has  jumps, preventing trajectorial tightness (at least in the $J1$-topology).  Following the idea of Kurtz \cite{kurtzaveraging} and as developed in Méléard and Tran \cite{meleardtransuperage} and Gupta et al. \cite{guptametztran},  a  weaker topology consists  in forgetting the process point of view and considering the measure $X^K_{Kt}(dx)dt$ in $\mathcal{M}_F([0,T]\times \X)$ embedded with the topology of weak convergence. This convergence in the sense of occupation measures strengthen the result of Theorem \ref{thm:TSS} but in a topology weaker than the Skorohod topology.

\noindent
To achieve this, as in Collet et al. \cite{CMM}, we first introduce the $M_1$-topology on $\D([0,T],\R_+)$. It is weaker than the usual $J_{1}$-topology and allows monotonous processes with jumps tending to $0$ to converge to processes with jumps (see Skorokhod \cite{skorohod}). For a c\`adl\`ag function $h$ on $[0,T]$, the continuity modulus for the $M_{1}$-topology  is given by
\be
\label{modulus}
w_{\delta}(h) = \sup_{\stackrel{\scriptstyle 0\leq t_{1}\leq t\leq t_{2}\leq T;} {0\leq t_{2}-t_{1}\leq \delta}} d(h(t), [h(t_{1}), h(t_{2})]).
\ee
Note that if the function $h$ is monotone, then $w_{\delta}(h) = 0$.

\begin{prop} Let us consider a continuous and monotonous function $g$.
Then, under Assumptions \ref{hypo} and \ref{IIF}, the process $(R^K_{t}, t\in [0,T])$ defined by
$$R^K_{t} = \int g(x)  X^K_{Kt}(dx)$$
converges in law in the sense of the Skorohod  $M_{1}$-topology to the process $(R_{t},  t \in [0,T])$ where $R_{t} =  \widehat{n}_{Y_{t}}\,g(Y_{t})$.
\end{prop}

\bi
\begin{proof} Assume that $g$ is non-decreasing.
From Theorem \ref{thm:TSS},
finite dimensional distributions of $(R^{K}_{t}, t\in [0,T])$ converge to those of $(\widehat{n}_{Y_{t}}\,g(Y_{t}),  t \in [0,T])$. By  \cite{skorohod} Theorem 3.2.1, it remains to prove that for all $\eta>0$,
$$\lim_{\delta\to 0} \limsup_{K\to \infty} \mathbb{P}(w_{\delta}(R^{K}_{.}) >\eta) = 0,
$$
where $w_{\delta}$ has been defined in \eqref{modulus}.

\noindent The mutation rate in $(R^{K}_{t}, t\in [0,T])$ being bounded, the probability that two mutations occur within a time less that $\delta$ is $o(\delta)$. It is therefore enough to study the case where there is  at most one mutation in the time interval $[0,\delta]$. Following Champagnat \cite{champagnat06}, the path of $R^K$ can be decomposed into several subpaths, each of them being closed to a large population deterministic measure-valued function $\xi$ (See Proposition \ref{prop:gdepop} in the appendix) with a probability tending to 1.
Away from invading mutations and for a trait-monomorphic population with trait $x$, $\langle \xi_{Kt},g\rangle = g(x) n_{Kt}(x)$ where $n_{.}(x)$ is the solution of the logistic equation \eqref{eq1:logistique}. We can easily check that $t\to n_{t}(x)$ converges monotonously to its stable equilibrium $\widehat n_{x}$ and then $\langle \xi_{Kt},g\rangle$ is monotonous and  the modulus of continuity tends to $0$. Around an invading mutant $y$ , $\langle \xi_{Kt},g\rangle$ is close to $n_{Kt}(x) g(x) + n_{Kt}(y) g(y)$ where $(n_{t}(x), n_{t}(y)) $ is solution of the Lotka-Volterra system \eqref{largepop-dimorphique} with an initial condition close to $(\widehat n_{x},0)$. The mutant $y$ invades if the fitness function $f(y;x)$ is positive (and $f(x;y)$ is negative). From Assumption \ref{IIF}, an easy study of the Lotka-Volterra system (see for example the appendix in Champagnat \cite{champagnatthesis}, Figure  (b) p.187), shows that either $n_{t}(x)$ and $n_{t}(y)$ are increasing or $\dot n_{t}(x) <0 ;\ \dot n_{t}(y) >0$. In that case we can write $${d\over dt}(n_{t}(x) g(x) + n_{t}(y) g(y)) = (g(y)-g(x))\dot n_{t}(y) +g(x) (\dot n_{t}(y)-\dot n_{t}(x))\geq 0$$ since $g$ is monotonous. Therefore  the function $\langle \xi_{Kt},g\rangle$ is increasing for $K$ large enough and the same conclusion holds.
\end{proof}

\begin{cor}\label{corollary:occup-X}
The sequence of random measures
$  X^K_{Kt} (dx) dt$ converges in law to the random measure $  \widehat{n}_{Y_{t}}\, \delta_{Y_{t}}(dx) dt$ in $\mathcal{M}_F([0,T]\times \X)$ embedded with the weak convergence topology.
\end{cor}

\begin{proof}  It is enough to prove the convergence in law of \\
$\int  h(t) e^{-q x}  X^K_{Kt} (dx) dt\,$ to $\,\int  h(t) e^{-q x} \hat{n}_{Y_{t}}\, \delta_{Y_{t}}(dx) dt$ for a measurable bounded  function $h$ and $q\in \mathbb{Q}$. In \cite{skorohod}, it is proved that if $x_{K}$ converges to $x$ in $\mathbb{D}([0,T],\mathbb{R})$  embedded with the $M_{1}$-topology, then for  $t$ outside a denumerable set, $x_{K}(t)$ converges to $x(t)$. Then  it follows by Lebesgue's Theorem that $\int_{0}^T H(t,x_{K}(t))dt$ converges to $\int_{0}^T H(t,x(t))dt$, as soon as $H$ is bounded and continuous. We apply this result to the process $(\int_{\X} e^{-q x}  X^K_{Kt} (dx), t\geq 0)$ and the function
$$H_{M}(t,y) = h(t) (y \wedge M),$$ for any $M>0$.  Estimate \eqref{estimeemoment} (with $p=1$)  allows to conclude.
\end{proof}

\subsubsection{Marker distribution in a new  adaptive trait mutant population}\label{sec:MarkDistrib}

In this section, we study the transition of the marker distribution  when a new mutant adaptive  trait appears in a monomorphic population with trait $x_0$. We consider this phenomenon at the ecological time scale and we prove that the fixation of the mutant trait creates a genetical bottleneck.

\me Let $K$ be fixed.  Initially we have  a trait monomorphic population with trait $x_0$ and a marker distribution $\pi^K(x_0,du)$. Then an individual $(x_{0},v)$ from this population gives birth to   an individual  with mutant trait $y$  and marker $v$ ($v$ has been chosen according to $\pi^K(x_0,du)$). We consider the process $(\nu^K_t ; t\geq 0)$  started at
\begin{align*}
\nu^K_0(dx,du)= & X^K_0(dx)\pi^K_0(x,du)\\
=& \frac{1}{K}\delta_{(y,v)}(dx,du)+\frac{N^K_0-1}{K}\delta_{x_0}(dx)\pi^K_0(x_0,du).
 \end{align*}


\begin{prop}\label{prop:hitchhiking}Under Assumptions \ref{hypo} and \ref{IIF}, let us consider a mutant $(y,v)$ appearing in a monomorphic population with trait $x_0$ and marker distribution $\pi^K_0(x_0,du)$. Let us assume that $f(y ; x_0)>0$, where the fitness function has been defined in \eqref{def:fitness}. There exists $\varepsilon>0$ such that for any sequence $(t_K ; K\in \N^*)$ with $\lim_{K\rightarrow +\infty} t_K/\log K=+\infty$ and $\lim_{K\rightarrow +\infty} t_K/K=0$ (for example $t_K=(\log K)^2$), we have
\begin{equation}
\lim_{K\rightarrow +\infty}\P\big(\langle \nu^K_{t_K},\ind_{y}\rangle>\varepsilon\big)=\frac{f(y ; x_0)}{b(y)} \mbox{ and }\lim_{K\rightarrow +\infty}\P\big(\langle \nu^K_{t_K},\ind_{y}\rangle=0\big)=1-\frac{f(y ; x_0)}{b(y)}.\label{extinction-survie}
\end{equation}
Further, for the marker distribution, we can prove that
\begin{equation}
\lim_{K\rightarrow +\infty}\P\big(\pi^K_{t_K}(y,du)=\delta_v(du)\big)=\frac{f(y ; x_0)}{b(y)}.\label{distrib-marker}
\end{equation}
\end{prop}

\me The equation \eqref{distrib-marker} tells us that when the mutant trait survives in the resident population of trait $x_0$, then by the time $t_K$ it needs to reach a non-negligible size, its marker distribution is still a Dirac mass at $y$. Additional comments are given after the proof.\\

\begin{proof} Properties \eqref{extinction-survie} have been proved in  \cite{champagnat06,champagnatferrieremeleard2} and depend only on the trait distribution. We consider test functions $\phi(x,u)$ of the form $\ind_{y}(x)g(u)$ with $g\in \Co^2(\U,\R)$ such that $\|g\|_\infty+\|g''\|_\infty\leq 1$. Starting from Proposition \ref{prop:nuK} and using Itô's formula with jumps, we obtain as soon as the population with trait $y$ survives,
\begin{align}
\int_{\U} g(u) \pi^K_{t_K}(y,du) = & \frac{\langle \nu^K_{t_K},\ind_{y}g\rangle}{\langle \nu^K_{t_K},\ind_y\rangle}\nonumber\\
= &   g(v)  + M^{K,g}_{t_K}
+  q_K (1-p_K)\, b(y) \int_0^{t_K} \left(1-\frac{1}{K\langle \nu^K_{s},\ind_{y}\rangle+1}\right)\label{etape7}\\
& \hspace{3cm} \times  \int_\U \pi^K_s(y,du)  \, \int_\U \big(g(u+h)-g(u)\big)G_K(u,dh) \ ds \nonumber
\end{align}
where $M^{K,g}$ is a square integrable martingale with previsible quadratic variation:
\begin{multline}
\langle M^{K,g}\rangle_{t_{K}} =  \frac{1}{K}\int_0^{t_{K}} ds \Big\{ \\
\begin{aligned}
& b(y)(1-q_K)(1-p_K)\frac{\langle \nu^K_s,\ind_y\rangle}{\big(\langle \nu^K_{s},\ind_{y}\rangle+\frac{1}{K}\big)^2}  \int_\U \big(g(u)-\langle \pi^K_s,g\rangle\big)^2\pi^K_s(y,du)   \\
+ & \big(d(y)+\eta(y) C*\nu^K_s(y) \big) \frac{\langle \nu^K_s,\ind_y\rangle}{\big(\langle \nu^K_{s},\ind_{y}\rangle-\frac{1}{K}\big)^2}\int_\U \big(g(u)-\langle \pi^K_s,g\rangle\big)^2\pi^K_s(y,du)\\
+&  b(y) q_{K}(1-p_K)  \frac{\langle \nu^K_s,\ind_y\rangle}{\big(\langle \nu^K_{s},\ind_{y}\rangle+\frac{1}{K}\big)^2} \int_\U\pi^K_s(y,du)  \int_\U G_K(u,dh) \big( g(u+h)-\langle \pi^K_s,g\rangle\big)^2  \bigg\}.\label{crochetetape7}
\end{aligned}\end{multline}

\me The third term in the right hand side of \eqref{etape7} is of order  $t_K/K$. Indeed thanks to \eqref{echelle} and \eqref{hyp-D},  it is upper bounded by
$$ {t_{K}\over K}\,\bar b\, \|Ag\|_{\infty} .$$
Similarly, the second term of \eqref{crochetetape7} tends to $0$ as ${t_{K}\over K}$.\\
The first term needs more attention. As soon as the mass $\langle \nu_s^K,\ind_{y}\rangle$ of the mutant population is of order 1, the variance of $M^{K,g}_{t_K}$ is in $t_K/K$ which tends to zero when $K\rightarrow +\infty$. However, since we start from 1 individual, we have to separate the time interval $[0,t_K ]$ into 2 parts. Let us introduce a sequence $(s_{K})$ such that  $s_{K} \leq t_{K}$ for any $K$ and
$$\log K \ll s_{K} \ll (\log K)^2.$$ Notice that $s_{K}$ can be equal to $t_{K}$.
Using Assumption \ref{IIF}, we can prove as in \cite[Lemma 3]{champagnat06} that there exists $\varepsilon_0>0$ such that
$$\lim_{K\to \infty} \P\Big(\forall s \in [s_{K}, t_{K}],\  \langle \nu^K_s,\ind_y \rangle  \geq \varepsilon_0\Big) = \frac{f(y;x_0)}{b(y)}.$$
It turns immediately out that
\begin{multline}
\E\left(\frac{\ind_{\{\forall s \in [s_K,t_K],\ \langle \nu^K_{s},\ind_y\rangle >0\}}}{K}\int_{s_{K}}^{t_{K}} \frac{b(y)+d(y)+\eta(y)C *\nu^K_s(y)}{\langle \nu^K_{s},\ind_{y}\rangle} \int_\U \big(g(u)-\langle \pi^K_s,g\rangle\big)^2\pi^K_s(y,du)  ds\right) \\ \leq C\, {t_{K}\over K}.\end{multline}
Before time $s_{K}$, the population size with trait $y$ is not large enough and $ \frac{1}{K\,\langle \nu^K_s,\ind_y\rangle} $ can only be upper bounded by $1$. Therefore we have to control the expectation of the variance of $g$ under $\pi^K_{s}$. The expected number of marker mutations at time $s$ along a lineage  is $s q_{K}$ and the variance of such mutation is bounded by
$\|g\|_\infty^2=\sup\{g(h)^2, h\in {\cal U}\}$.
Then
\begin{equation}
\E\left(\int_\U \big(g(u)-\langle \pi^K_s,g\rangle\big)^2\pi^K_s(y,du)\right)
\leq  s\, q_{K} \,\|g\|_\infty^2,\label{majo_transition}
\end{equation}
and
\ben
\E\left(\frac{1}{K}\int_{0}^{s_{K}}  \frac{b(y)+d(y)+\eta(y)C *\nu^K_s(y)}{\langle \nu^K_s,\ind_y\rangle} \int_\U \big(g(u)-\langle \pi^K_s,g\rangle\big)^2\pi^K_s(y,du)   ds\right) \leq C\, {(s_{K})^2 r_K \over K^2}.\een
\noindent
This concludes the proof.
\end{proof}

\begin{rem}For $q_K=1/\sqrt{K}$, let us notice that the rate of appearance of mutant markers in a population of size $K$ is of order $Kq_K=\sqrt{K}$ which does not tend to zero. This means that many mutant markers appear in the population of trait $y$ during the $t_K$ time interval following the first mutant $(y,v)$. However, heuristically, since in a tree the mass is concentrated around the leaves, the mutants do not appear with the same probability along the time interval and mutations are mostly observed after the time $s_K$ when the mutant population $(y,v)$ is already large. Moreover, using that the marker mutation step and/or marker mutation frequency is small we obtain that the mutant markers remain in negligible proportion between $s_K$ and $t_K$.
\end{rem}

\subsection{Convergence of the marker distribution process in a trait-monomorphic population}\label{section:convergence}

\me
For $K\in \N^*$, we introduce, as in \cite{champagnat06}, the following sequence of stopping times $\tau^K_k$ and $\theta^K_k$:
\begin{align*}
& \tau^K_0=  0, \qquad \theta^K_0=0\\
& \tau^K_{k+1}=\inf\{t>\tau^K_k,\ \card \big(\supp(\bar{X}^K_t)\big)=\card \big(\supp(\bar{X}^K_{t_-})\big)+1\}\\
& \theta^K_{k}=\inf\{t>\tau^K_k,\ \card \big(\supp(\bar{X}^K_t)\big)=1\}.
\end{align*}The times $\tau^K_k$'s are the times of appearance of the successive mutant traits in the population and the $\theta^K_k$'s are the times at which the population returns to monomorphic state. These times are possibly infinite, if the corresponding sets are empty. It has been proved in \cite{champagnat06} that for $t_K$ be such that $\lim_{K\rightarrow +\infty} t_K/\log(K)=+\infty$ and $\lim_{K\rightarrow +\infty} K t_K  =0$,
\begin{equation}\lim_{K\rightarrow +\infty}\P\Big(\forall k\geq 0, \tau^K_k\wedge KT \leq \theta^K_{k} \wedge KT \leq \big(\tau^K_k+t_K\big)\wedge KT\leq\tau^K_{k+1} \wedge KT\Big)=1.\label{etape8}
\end{equation}


\begin{prop}\label{prop:FV}Take the process $(\nu^K_{Kt} ; t\in [0,T])$ started with the monomorphic initial condition $\nu_0^K(dx,du)=n_0^K \delta_{(x_0,u_0)}(dx,du)$, where $\lim_{K\rightarrow +\infty}n_0^K= \widehat{n}_{x_0}>0$   and \\ $\sup_{K\in \N^*} \E((n^K_0)^3)<+\infty$.

\noindent (i) In the trait-mutation time scale, the time of first mutation converges in distribution as follows:
\begin{equation}
\lim_{K\rightarrow +\infty} \tau_1^K/K =\tau_1,\label{limitetau1}
\end{equation}where $\tau_1$ is an exponential time with parameter $b(x_0)\widehat{n}_{x_0}$.\\
(ii) Let us consider the processes $(\pi^K_{K(t\wedge \tau^K_1)} ; t\in [0,T])$ stopped at the time of first mutation. When $K\rightarrow +\infty$, this sequence converges in distribution in $\D([0,T],\mathcal{P}(\U))$ to the Fleming-Viot process $F^{u_0}(x_0,du)$ (see Definition \ref{def:FV}) and stopped at the independent exponential time $\tau_1$.\end{prop}

\begin{proof}First of all, the trait and marker mutations are independent. Thus, the stopping time $\tau_1^K$ is independent of the marker distribution $\pi^K(x_0,du)$. The results of Champagnat and coauthors \cite{champagnat06,champagnatferrieremeleard2} are unchanged and give \eqref{limitetau1}. Moreover, by \cite[Lemma 5.4]{champagnatferrieremeleard2}
\begin{equation}\lim_{K\rightarrow +\infty}\P\Big(\sup_{s\in [\log K, \tau^K_1]}\langle X^K_s, {\bf 1}\rangle \geq \frac{\widehat{n}_{x_0}}{2}\Big)=1.
\label{hyp-mass}\end{equation}

\noindent Let $\phi\in \Co(\U,\R)$. Since the population is trait-monomorphic with trait $x_0$, then\\  $\langle \pi^K_{Kt}(x_0,du),\phi(u)\rangle = \frac {\langle \nu^K_{Kt},\phi\rangle}  {\langle \nu^K_{Kt}, {\bf 1}\rangle}$. Thus, from Proposition \ref{prop:nuK} and It\^o's formula, we get that in the time scale $K t$
\begin{multline}\label{etape4}
\langle \pi^K_{K(t\wedge \tau^K_1)}(x_0,.),\phi\rangle
=  \langle \pi^K_0(x_0,.),\phi\rangle + H^{K,\phi}_{K(t\wedge \tau^K_1)}
+  b(x_0) q_K(1-p_K)\\\times \int_0^{t\wedge \tau^K_1}
\Big(1-\frac{1}{K\langle \nu^K_{Ks},1\rangle+1}\Big) \frac{r_K}{K}\int_{\U\times \U} \big(\phi(u+h)-\phi(u)\big)G_K(u,dh) K\pi^K_{Ks}(x_0,du) ds
\end{multline}
where $ H^{K,\phi}$ is a square integrable martingale with quadratic variation is
\begin{multline}
\langle H^{K,\phi}\rangle_{K(t\wedge \tau^K_1)}=  \int_0^{t\wedge \tau^K_1} b(x_0)(1-q_K)(1-p_K)\frac{\langle \nu^K_{Ks},\ind_{x_0}\rangle}{\big(\langle \nu^K_{Ks},{\bf 1}\rangle+\frac{1}{K}\big)^2} \times \\
\begin{aligned}
 & \hspace{6cm} \int_{\U} \pi^K_{Ks}(x_0,du) \big(\phi(u)-\langle \pi^K_{Ks}(x_0,.),\phi\rangle \big)^2\\
 +  & \big(d(x_0)+\eta(x_0) C*X^K_{Ks}(x_0)\big)\frac{\langle \nu^K_{Ks},\ind_{x_0}\rangle}{\big(\langle \nu^K_{Ks},{\bf 1}\rangle-\frac{1}{K}\big)^2}\times \\
& \hspace{6cm} \int_{\U}  \big(\phi(u)-\langle \pi^K_{Ks}(x_0,.),\phi\rangle \big)^2\\
 + & b(x_0)q_K(1-p_K) \frac{\langle \nu^K_{Ks},\ind_{x_0}\rangle}{\big(\langle \nu^K_{Ks},{\bf 1}\rangle+\frac{1}{K}\big)^2}\int_\U \pi^K_{Ks}(x_0,du)  \big(\phi(u+h)-\langle \pi^K_{Ks}(x_0,.),\phi\rangle \big)^2
\end{aligned}
\label{crochetetape4}
\end{multline}



\me The computation shows that the order of the quadratic variation of $\pi^K_{t}$  is ${1\over K}$. Thus at  time scale $K t$, this order will be $1$. That justifies Assumption \eqref{echelle} for $p_{K}$ which is the only choice to get a non degenerate diffusive  limit.

\me Let us introduce a process $(\widetilde{\pi}^{K,x_0}_t(du), t\geq 0)$ coupled with $(\pi^{K}_t(x_0,du),t\geq 0)$, on the same probability space and driven by the same Poisson point measures, that satisfies the following properties. The dynamics of $\widetilde{\pi}^{K,x_0}_.(du)$ is given by \eqref{etape4}-\eqref{crochetetape4} but without the stopping times $\tau_1^K$ and we have that $\forall t\geq 0,\ \pi^{K}_{K(t\wedge \tau_1^K)}(x_0,du)=\widetilde{\pi}^{K,x_0}_{K(t\wedge \tau_1^K)}(du)$. In a nutshell, $(\widetilde{\pi}^{K,x_0}_t,t\geq 0)$ corresponds to the process $(\pi^{K}_t(x_0,.),t\geq 0)$ that is obtained by setting the trait mutation kernel to the Dirac mass at 0.

\me
\noindent Thanks to \eqref{hyp-D}, \eqref{hyp-mass} and using that $\widetilde{\pi}^{K,x_0}$ is a probability-valued process, \eqref{etape4} and \eqref{crochetetape4} imply that for any $\phi\in \Co(\U,\R)$, the distribution sequence of $(\langle \widetilde{\pi}^{K,x_0}_{K.},\phi\rangle ; K\in \N^*)$ is uniformly tight in $\D([0,T],\R)$. By Roelly's criterion \cite[Theorem 2.1]{roelly}, this implies the uniform tightness of the sequence of the laws of $(\widetilde{\pi}^{K,x_0}_{K.} ; K\in \N^*)$ in $\D([0,T],\mathcal{P}(\U))$.\\
Let us consider a limiting value $(\bar{\pi}_t(du) ; t\in [0,T])$ of the tight sequence and a subsequence, again denoted by $\widetilde{\pi}^{K,x_0}_{K.}(du)$,  that converges to $\bar{\pi}_.(du)$.  By Assumption \eqref{hyp-mass} and since individuals have weight $1/K$, the limiting laws only charge $\Co([0,T],\mathcal{P}(\U))$.

\me \noindent It remains to identify $\bar{\pi}_.(du)$. Let $0<s<t<T$, let $k\in \N$ and $0<s_1\leq \cdots s_k<s<t$, let $\phi_1,\cdots \phi_k$ be bounded continuous function on $\mathcal{P}(\X\times \U)$ and $\phi\in \Co(\U,\R)$. We define the following bounded functional on $\D([0,T],\mathcal{P}(\U))$
\begin{align*}
\Psi_{s,t}(Y)= & \phi_1(Y_{s_1})\cdots \phi_k(Y_{s_k})\Big\{\langle Y_t,\phi\rangle - \langle Y_s,\phi\rangle -\int_s^t du \int_{\U} Y_u(du)b(x_0)A\phi(u)\Big\}
\end{align*}
On the one hand, using \eqref{etape4}, we obtain that
\begin{align*}
\Psi_{s,t}(\widetilde{\pi}^{K,x_0}_{K.})=\phi_1(\widetilde{\pi}^{K,x_0}_{Ks_1})\cdots \phi_k^K(\widetilde{\pi}^{K,x_0}_{Ks_k})\Big\{H^{K,\phi}_{Kt}-H^{K,\phi}_{Ks}
 +  \varepsilon_{Kt}^K
\Big\}
\end{align*}where
\begin{align}
\varepsilon^K_{Kt}= & \int_0^{t} ds  \int_{\U} \widetilde{\pi}^{K,x_0}_{Ks}(du) \Big[b(x_0) \frac{r_K}{K}\int_{\U} \big(\phi(u+h)-\phi(u)\big)G_K(u,dh)\Big]\nonumber\\
- &  \int_0^{t} ds  \int_{\U} \widetilde{\pi}^{K,x_0}_{Ks}(du) b(x_0) A \phi(u)\label{etape5}
\end{align}tends to 0 in $\mathbb{L}^1$ when $K\rightarrow +\infty$.
Thus,
\begin{equation}
\lim_{K\rightarrow +\infty}\E\Big(\Psi_{s,t}(\widetilde{\pi}^{K,x_0}_{K.})\Big)=0.
\end{equation}
On the other hand, using \eqref{estimeemoment} and the convergence of $(\widetilde{\pi}^{K,x_0}_{K(.\wedge \tau^K_1)}(du) ; K\in \N^*)$ to $\bar{\pi}\in \Co([0,T],\mathcal{M}_F(\U))$, we get
\begin{align}
\E\big(\Psi_{s,t}(\bar{\pi})\big)=\lim_{K\rightarrow +\infty}\E\big(\Psi_{s,t}(\widetilde{\pi}^{K,x_0}_{K.}(du))\big).\label{etape2}
\end{align}
This shows that $\E\big(\Psi_{s,t}(\bar{\pi})\big)=0$ and hence the process $M^{x_0}(\phi)$ defined in \eqref{PBM-FV} is a martingale obtained as the uniform limit in time of $H^{K,\phi}_{Kt}$, when $K\rightarrow +\infty$. Moreover, the  bracket \eqref{crochetetape4} converges to
\begin{multline}
\int_0^t \frac{b(x_0)+d(x_0)+\eta(x_0) \widehat{n}_{x_0}}{\widehat{n}_{x_0}}\int_{\U} \bar{\pi}(du)\Big[
\big(\phi(u)-\langle \bar{\pi}_{s},\phi\rangle \big)^2 \Big] ds \\
=  \int_0^t \frac{2b(x_0)}{\widehat{n}_{x_0}} \Big[ \langle \bar{\pi}_s,\phi^2\rangle -\langle \bar{\pi},\phi\rangle ^2 \Big] ds.
\label{etape6}
\end{multline}
Indeed, the integral in \eqref{crochetetape4} can be separated into two integrals, one between $0$ and $\frac{\log K}{K} \wedge t \wedge \tau_1^K$ and the other between $\frac{\log K}{K} \wedge t \wedge \tau_1^K$ and $t \wedge \tau_1^K$. The second integral converges to \eqref{etape6}, but some caution is needed for the first integral since the ratios $\langle \nu^K_{Ks},\ind_{x_0}\rangle/\big(\langle \nu^K_{Ks},{\bf 1}\rangle \pm \frac{1}{K}\big)^2$ are of order $K$. Using the same arguments as for \eqref{majo_transition}, we can upper bound the integral between $0$ and $\frac{\log K}{K} \wedge t \wedge \tau_1^K$ by
$$ C K \int_0^{\log K/K} s q_K ds=C\frac{(\log K)^2}{K} q_K\rightarrow_{K\rightarrow +\infty} 0.$$

Using Theorem 3.12 p. 432 of \cite{jacod}, that provides the convergence of $H^{K,\phi}_{K.}$ to the solution of the martingale problem \eqref{PBM-FV}-\eqref{crochet-PMB-FV} with $x=x_{0}$.

 \noindent By the independence of $\widetilde{\pi}^{K,x_0}_{K.}(du)$ and $\tau_1^K$, $\tau_1$ is independent of $\bar{\pi}_.(du)$ and $\bar{\pi}_{.\wedge \tau_1}=\pi^{u_0}_{.\wedge \tau_1}(x_0,du)$. This concludes the proof.
\end{proof}

\subsection{Conclusion}

Using Theorem \ref{thm:TSS}, Proposition \ref{prop:hitchhiking} and Proposition \ref{prop:FV}, we prove the first part of Theorem \ref{thm:SFVP}, for the convergence in finite distribution. Let us prove the convergence in the sense of occupation measures.

\begin{cor}\label{lem2}
The family $(\nu^K_{Kt}(dx,du)\ dt, K\in \N^*)$ is uniformly tight in $\mathcal{M}_F(\X \times \mathcal{U}\times [0,T])$ embedded with the weak convergence topology and converges in distribution to the measure $V_t(dx,du)dt$, where $V$ is defined in Theorem \ref{thm:SFVP}.
\end{cor}

\begin{proof}
Since the space $\X\times \U\times [0,T]$ is compact, it is sufficient to prove that
$$\sup_{K\in \N^*} \E\big(\int_0^T \langle \nu_{Kt}^K,1\rangle dt\big)<+\infty$$ which is a consequence of Fubini's theorem since
\begin{align*}
\E\big(\int_0^T \langle \nu_{Kt}^K,1\rangle dt\big) \leq T \sup_{K\in \N^*, t\in \R_+}\E\big(\langle \nu^K_t,1\rangle\big).
\end{align*} Estimate \eqref{estimeemoment} concludes the proof of tightness.\\


\noindent Let us now consider continuous functions $\phi\in \Co(\X\times [0,T],\R)$ and $g\in \Co(\U,\R)$, where the stopping times $\tau^K_k$ and $\theta^K_k$ have been introduced in Section \ref{section:convergence}. Then
\begin{multline}
\int_0^T \int_{\X\times \U} \phi(x,t)g(u)\nu^K_{Kt}(dx,du)\ dt =  \int_0^T \langle \pi^K_{Kt}(x,.),g\rangle \phi(x,t)X^K_{Kt}(dx)dt\\
=  \sum_{k\geq 0}  \int_{\theta^K_k}^{\tau^K_{k+1}} \langle \pi^K_{Kt}(x,.),g\rangle \phi(x,t)X^K_{Kt}(dx)dt
 +    \int_{\tau^K_{k+1}}^{\theta^K_{k+1}} \langle \pi^K_{Kt}(x,.),g\rangle \phi(x,t)X^K_{Kt}(dx)dt,\label{etape9}
\end{multline}The limit \eqref{etape8} implies that the second term of the right hand side of \eqref{etape9} converges to 0. Given $X^K$, the processes $(\pi^K_{Kt}(x,.) ; t\in [\theta^K_k, \tau^K_{k+1}))$, for $k\geq 0$, in the first term of the r.h.s. of \eqref{etape9} are independent and, by Proposition \ref{prop:FV}, they converge in distribution in $\D([0,T],\R)$ to the Fleming-Viot processes \eqref{PBM-FV}-\eqref{crochet-PMB-FV} with the initial conditions described by the jumps of the extended TSS $(Y,U)$.
Corollary \ref{corollary:occup-X} and dominated convergence theorem allows us to conclude the proof.
\end{proof}




{\footnotesize
\providecommand{\noopsort}[1]{}\providecommand{\noopsort}[1]{}\providecommand{\noopsort}[1]{}\providecommand{\noopsort}[1]{}

}
\appendix

\section{Limit theorems for the trait and marker distributions in the ecological time scale}

If we let $K\rightarrow +\infty$ without changing the time scale, we obtain:
\begin{prop}
\label{prop:gdepop}
Assume that the sequence $(\nu^K_0(dx,du) ; K\in \N^*)$ converges in probability to the measure $\xi_0(dx,du)$ when $K\rightarrow +\infty$. Then the sequence of processes $(\nu^K_t(dx,du) ; t\geq 0)_{K\in \N^*}$ converges in $\mathbb{D}(\R_+,\mathcal{M}_F(\X\times \U))$ to the deterministic process $\xi\in \Co(\R_+,\mathcal{M}_F(\X\times \U))$ defined for every $\phi(x,u)\in \Co(\X\times \U,\R)$ by:
\begin{equation}
\langle \xi_t,\phi\rangle=\langle \xi_0,\phi\rangle+ \int_0^t \int_{\X\times \U}\big(b(x)-d(x)-\eta(x)C*\xi_s(x) \big)\phi(x,u)\xi_s(dx,du)\ ds.
\end{equation}
\end{prop}

\begin{proof}The proofs proceed in a classical tightness-uniqueness way (cf. Fournier-M\'el\'eard \cite{fourniermeleard}).
\end{proof}

\noindent Notice that no mutation can be seen at this scale. To see  the trait  mutations, we have to consider the process at the
mutation scale $Kt$ (cf. \cite{champagnatferrieremeleard2}).

\paragraph{Trait-monomorphic case}

From Proposition \ref{prop:gdepop}:
\begin{cor}Assume that the initial population is trait-monomorphic
$
\nu^K_0(dx,du)=n_0^K \delta_{x_0}(dx) \pi^K_0(x_0,du)$
where $\lim_{K\rightarrow +\infty} n_0^K = n_0$ and $\lim_{K\rightarrow +\infty}\pi^K(x_0,du)= \pi_0(x_0,du)$ in probability. Then the sequence $(\nu^K ; K\in \N^*)$ converges, in probability and uniformly on every compact time interval $[0,T]$ with $T>0$, to $(\nu_t(dx,du)=n_t(x) \delta_{x_0}(dx)\pi_0(x_0,du) ; t\geq 0)$
where $n_t(x)$ is the deterministic solution of the logistic equation
\begin{equation}
\frac{dn_t}{dt} =\big(b(x)-d(x)-\eta(x) C(0)n_t(x)\big) n_t (x)
\end{equation}
which converges when $t$ tends to infinity to
\begin{equation}
\widehat{n}_x =\frac{
b(x)-d(x)}{\eta(x)C(0)}.
\end{equation}
\end{cor}

\begin{proof}For the part of the proof dealing with $\xi$, we refer to \cite{champagnatferrieremeleard2}.
\end{proof}

\paragraph{Trait-dimorphic case}

From Proposition \ref{prop:gdepop}:
\begin{cor}Assume that the initial population is trait-dimorphic
$$\nu^K_0(dx,du)=n_0^K(x_1) \delta_{x_1}(dx) \pi^K_0(x_1,du)+n_0^K(x_2) \delta_{x_2}(dx) \pi^K_0(x_2,du)$$
where $n^K_0(x)$ is the number of individuals with trait $x$ renormalized by $K$. We assume that for $x\in \{x_1,x_2\}$, $\lim_{K\rightarrow +\infty} n_0^K(x) = n_0(x)>0$ and $\lim_{K\rightarrow +\infty}\pi^K(x,du)= \pi_0(x,du)$ in probability.
 We also assume that $(x_1, x_2)$ satisfies the Assumption \ref{IIF}.\\
Then the sequence $(\nu^K ; K\in \N^*)$ converges, in probability and uniformly on every compact time interval $[0,T]$ with $T>0$, to\\
 $(\nu_t(dx,du)=n_t(x_1) \delta_{x_1}(dx)\pi_0(x_1,du)+n_t(x_2) \delta_{x_2}(dx)\pi_0(x_2,du) ; t\geq 0)$\\
 where $(n_t(x_1), n_t(x_2))$ solves the system
\begin{align}
& \frac{dn_t(x_1)}{dt} =\big(b(x_1)-d(x_1)-\eta(x_1) (C(0)n_t(x_1)+C(x_1-x_2)n_t(x_2))\big) n_t (x_1)\nonumber\\
& \frac{dn_t(x_2)}{dt} =\big(b(x_2)-d(x_2)-\eta(x_2) (C(x_2-x_1)n_t(x_1)+C(0)n_t(x_2))\big) n_t (x_2).\label{largepop-dimorphique}
\end{align}
whose only stable equilibrium is $(0,\widehat{n}_{x_2})$, with $\widehat{n}_{x_2}$ defined in \eqref{nchap}.
\end{cor}

\noindent It can be seen that the conditional distributions of the marker, given the trait $x_1$ or $x_2$ remain constant.

\me
\begin{proof}The convergence in large population of $(\nu^K ; K\in \N^*)$
is a consequence of Proposition \ref{prop:gdepop}. 
%
\end{proof}

\end{document}